\documentclass[12pt]{article}
\usepackage{amsmath,bbm,amssymb}
\usepackage{amsthm}
\usepackage{comment}
\usepackage{graphicx, graphics}
\usepackage{textcomp}
\usepackage{amsfonts}

\usepackage{psfrag,pst-node,subfigure,rotating, amsmath, bbm, amsthm, amssymb, amsthm, setspace, picture, epsfig, amsfonts, upgreek}

\usepackage{multirow}

\usepackage{arydshln}

\newtheorem{assumption}{Assumption}
\newtheorem{corollary}{Corollary}
\newtheorem{theorem}{Theorem}
\newtheorem{lemma}{Lemma}
\newtheorem{remark}{Remark}
\newtheorem{proposition}{Proposition}

\newtheorem{definition}{Definition}

\numberwithin{assumption}{section} \numberwithin{corollary}{section} \numberwithin{theorem}{section} \numberwithin{lemma}{section} \numberwithin{definition}{section}
\numberwithin{proposition}{section} \numberwithin{remark}{section}

\def\v{\mathbf}\def\m{\mathbf}\def\vg{\boldsymbol}\def\s{\mathcal}
 \def\X{{\mathbb{X}}}

\begin{document}
\begin{center}
    {\Large\bf Better Solution Principle: A Facet of Concordance between
Optimization and Statistics}
\\[2mm] {\large Shifeng Xiong}
\\ Academy of Mathematics and Systems Science \\Chinese Academy of Sciences, Beijing 100190\\xiong@amss.ac.cn
\end{center}
\vspace{-10mm}

\vspace{1cm} \noindent{\bf Abstract}\quad Many statistical methods require solutions to optimization problems. When the global solution is hard to attain, statisticians
always use the better if there are two solutions for chosen, where the word ``better" is understood in the sense of optimization. This seems reasonable in that the
better solution is more likely to be the global solution, whose statistical properties of interest usually have been well established. From the statistical perspective,
we use the better solution because we intuitively believe the principle, called \emph{better solution principle} (BSP) in this paper, that a better solution to a
statistical optimization problem also has better statistical properties of interest. BSP displays some concordance between optimization and statistics, and is expected
to widely hold. Since theoretical study on BSP seems to be neglected by statisticians, this paper aims to establish a framework for discussing BSP in various statistical
optimization problems. We demonstrate several simple but effective comparison theorems as the key results of this paper, and apply them to verify BSP in commonly
encountered statistical optimization problems, including maximum likelihood estimation, best subsample selection, and best subset regression. It can be seen that BSP for
these problems holds under reasonable conditions, i.e., a better solution indeed has better statistical properties of interest. In addition, guided by the BSP theory, we
develop a new best subsample selection method that performs well when there are clustered outliers.


\vspace{2mm} \noindent{{\bf KEY WORDS:} Best subsample selection; Best subset regression; Combinatorial optimization; Global optimization; Large-scale optimization;
Likelihood principle; Robust estimation; Separation property; Variable selection.}

\newpage

\section{Introduction}\label{sec:intro}
\hskip\parindent \vspace{-0.8cm}

Many statistical methods require solutions to optimization problems. A notable example is maximum likelihood estimation, whose objective is to maximize the likelihood
function. Below is a brief description of some statistical methods that rely on optimization problems in various statistical areas.

\begin{description}

\item[$\bullet$] \emph{Maximum likelihood and related methods}: The maximum likelihood method can be used for parametric models and has good statistical properties
under regularity conditions. An extension of this method is M-estimation (Huber 1981), which obtains estimators by minimizing a general class of functions with respect
to the parameter. A corresponding method for nonparametric models is empirical likelihood (Owen 2001), which requires maximizing the empirical likelihood function.

\item[$\bullet$] \emph{Model fitting and selection}: In regression analysis, the parameter of the regression model needs to estimate for yielding a good fit to the data.
For this purpose, methods that minimize criteria which justify the goodness of fit are used such as the least squares method. Smoothing spline regression (Wahba 1990)
and local polynomial regression (Fan and Gijbels 1996) can be viewed as two variants of the least squares method in nonparametric settings. When model selection is
concerned, regularized regression methods, which minimize the regularized criteria to produce sparse estimators, can be used such as best subset regression (the
$\ell_0$-norm regularized method) and the lasso (the $\ell_1$-norm regularized method); see e.g., Hastie, Tibshirani, and Friedman (2008).

\item[$\bullet$] \emph{Multivariate analysis}: Many problems in multivariate statistical analysis involve projections of the data into a lower dimensional space.
Principal component analysis, canonical correlation analysis, and Fisher's discrimination are well known examples (Anderson 2003). Optimization problems over a
multi-dimensional sphere need to solve to find the projections.

\item[$\bullet$] \emph{Bayesian statistics}: A method in Bayesian point estimation is to use the posterior mode, which is the maximum of the posterior density
(Gelman et al. 2004). Maximum likelihood estimation can be viewed as a special case of this method.

\item[$\bullet$] \emph{Robust estimation}: Besides the M-estimate, popular robust estimates which can be formulated as optimization problems
include the least trimmed squares estimate, the S-estimate, and the minimum covariance determinant estimate, among others (Maronna, Martin, and Yohai 2006). Optimization
methods are also ubiquitous in computing depth functions (Zuo and Serfling 2000), which are useful to define multivariate median and trimmed mean.

\item[$\bullet$] \emph{Design of experiments}: A number of experimental designs are constructed by optimizing certain criteria. An example is
the minimum aberration criterion in fractional factorial designs (Wu and Hamada 2009). For continuous factors, optimal designs (Atkinson, Donev, and Tobias 2007) are
derived by optimizing model-based criteria, and space-filling designs (Fang, Li, and Sudjianto 2006) correspond to geometric or discrepancy criteria.

\item[$\bullet$] \emph{Statistical learning}: This area seriously utilizes computation for statistical inference. Many important methods such as
support vector machine (Scholkopf and Smola 2002) and boosting (Freund and Schapire 1997) are based on minimizing loss functions. In addition, (regularized) maximum
likelihood estimation is commonly used for graphical models (Wainwright and Jordan 2008). Cross-validation, which minimizes the empirical prediction error, is ubiquitous
in various methods to select tuning parameters (Hastie, Tibshirani, and Friedman 2008).

\end{description}

The above description, although far from thorough, indicates that optimization plays a vital role in modern statistics. In the meanwhile, statisticians have to face the
common difficulty in the optimization community, i.e., it is often extremely hard to obtain the global solution to a nonconvex optimization problem. A number of global
optimization algorithms have been proposed, including the simulated annealing algorithm (Kirkpatrick, Gelatt, and Vecchi 1983) and the genetic algorithm (Dorsey and
Mayer 1995). However, they can attain the global solution only in the probabilistic sense, and often take an unrealistically long time to approach it in practice (Lundy
and Mees 1986). When handling large-scale data, the problem of multiple extrema becomes more serious. In fact, for such cases, it is also hard to obtain the solution to
a convex optimization problem due to the unaffordable computational time and memory (Tibshirani et al. 2012; Ma, Mahoney, and Yu 2013). Another difficulty from the
problem of multiple extrema is that we can rarely know whether a solution at hand is the global solution (Gan and Jiang 1999).

When the global solution is hard to attain and/or to verify, statisticians always take the solution whose objective value is as small as possible (for minimization
problems) as the final solution. In other words, for two solutions, we always use the ``better" one, where the word ``better" should be understood in the sense of
optimization. This seems reasonable in that the better solution is more likely to be the global solution, whose statistical properties of interest usually have been well
established. From the statistical perspective, we use the better solution because we intuitively believe the principle, called \emph{better solution principle} (BSP) in
this paper, that a better solution to a statistical optimization problem also has better statistical properties of interest (closer statistical properties to those the
global solution has). This principle shows some concordance, or monotonicity, between optimization and statistics, and is expected to widely hold. Strictly speaking, a
better solution can safely be used only after the corresponding BSP is verified. However, it is surprising that statisticians seem to neglect this problem, although we
have actually made decisions following BSP ever since complex optimization problems appeared in statistics. To the best of the author's knowledge, no paper has formally
discussed BSP. For example, in the maximum likelihood problem, it is not clear to us whether a better solution with greater likelihood has higher estimation accuracy.
Fairly recently, Xiong (2014) introduced the \emph{better-fitting better-screening} rule when discussing variable screening in high-dimensional linear models. This rule
tells us that a subset with smaller residual sums of squares possesses better asymptotic screening properties, i.e., is more likely to include the true submodel
asymptotically. Here such a subset can be viewed as a better solution to the $\ell_0$-norm constrained least squares problem. Therefore, the better-fitting
better-screening rule is actually the BSP for this problem. In this paper, we aim to establish a relatively general framework for discussing BSP in various statistical
optimization problems.

The rest of the paper is organized as follows. We first present examples where BSP immediately holds in Section \ref{sec:ke}. Such examples widely exist in experimental
designs. They can help us understand BSP and the reason why we introduce the theorems in the following text. In Section \ref{sec:th}, we demonstrate several comparison
theorems which state that a better solution is more likely to have good statistical properties if the optimization problem possesses certain separation properties. These
theorems, which look very simple and understandable, are effective to establish BSP in a general setting. Sections \ref{sec:smle}-\ref{sec:bsr} apply these results to
several statistical optimization problems, including maximum likelihood estimation, best subsample selection, and best subset regression. The latter two problems are
often combinatorial optimization problems, which are more difficult to solve than continuous problems from the optimization perspective. Here best subsample selection is
referred to as the method of selecting the best part of observations to make inferences in the presence of outliers, and the minimum covariance determinant estimate and
least trimmed squares estimate are instances of estimates based on it. We can see that BSP for these problems holds under reasonable conditions, i.e., a better solution
indeed has better statistical properties of interest. In Section \ref{sec:boe}, we develop a new best subsample selection method which can perform well when there are
clustered outliers. A robust estimator based on it having consistency even under contaminated models is of independent interest in robust statistics. Section
\ref{sec:dis} concludes with some discussion.

\section{Known examples where BSP holds}\label{sec:ke}
\hskip\parindent \vspace{-0.8cm}

Obvious examples where BSP holds exist in experimental designs derived by optimizing some criteria. We take the D-optimal design for example. Consider a regression
model\begin{equation*} y=\sum_{i=1}^d\theta_ir_i(\v{x})+\varepsilon,\end{equation*}where the control variable $\v{x}$ lies in a subset $\mathcal{D}$ of $\mathbb{R}^p$,
$r_i$'s are specified functions, $\vg{\theta}=(\theta_1,\ldots,\theta_d)'$ is the vector of unknown parameters, and $\varepsilon$ is the random error. Given the sample
size $n$, denote the experimental design by $\mathcal{P}=\{\v{x}_1,\ldots,\v{x}_n\}$. The information matrix of this design is
$\m{M}(\mathcal{P})=\m{R}(\mathcal{P})'\m{R}(\mathcal{P})$,
where$$\m{R}(\mathcal{P})=\left(\begin{array}{ccc}r_1(\v{x}_1)&\cdots&r_d(\v{x}_1)\\\vdots&\ddots&\vdots\\r_1(\v{x}_n)&\cdots&r_d(\v{x}_n)\end{array}\right).$$The
D-optimal design minimizes the generalized variance of the least squares estimate of $\vg{\theta}$, i.e., it is the solution to the optimization problem
\begin{equation}\label{do}\min_{\v{x}_i\in\mathcal{D}}\psi(\mathcal{P})=\big[\mathrm{det}(\m{M}(\mathcal{P}))\big]^{-1},\end{equation}where ``$\mathrm{det}$"
denotes determinant. For two designs $\mathcal{P}_1$ and $\mathcal{P}_2$ with $\psi(\mathcal{P}_1)\leqslant\psi(\mathcal{P}_2)$, it is clear that $\mathcal{P}_1$ leads
to a better estimator whose generalized variance is smaller. Therefore, if estimation accuracy (which is justified by generalized variance) is the statistical property
of interest, BSP for problem \eqref{do} holds.

The objective function in \eqref{do} itself is a statistical criterion, which does not involve any random variables.  This is the reason why BSP for \eqref{do}
automatically holds. The same conclusion can be drawn for other model-based optimal designs and minimum aberration designs. For criterion-based space-filling designs,
the geometric or discrepancy criteria used as objective functions seem not to have clear statistical interpretation. However, most of them relate to some desirable
statistical properties. For example, the criteria for constructing the minimax distance design (Johnson, Moore, and Ylvisaker 1990) and uniform design (Fang et al. 2000)
can act as factors in the upper bounds of some estimation errors (Wendland 2005; Niederreiter 1992). If such estimation errors are used to evaluate the corresponding
estimators, we can say that BSP holds.

Design of experiments is a pre-sampling work, and thus the objective functions used in this area do not involve the random sample (except for sequential designs that we
do not consider here). In statistical inference, we have to deal with objective functions depending on the sample, which makes the problem of BSP more complicated. From
the next section, we study whether BSP holds for sample-based optimization problems through introducing new definitions and theorems.

\section{The comparison theorems}\label{sec:th}
\hskip\parindent \vspace{-0.8cm}

Let $(\Omega, \mathfrak{F}, P)$ be a probability space. For simplicity, it is assumed that all sets and maps throughout this paper are measurable (with respect to
according $\sigma$-fields). For each $n\in \mathbb{N}$, the sample $\X_n$ of size $n$ is a map from $\Omega$ to a space $\mathcal{X}_n$. Based on $\X_n$, we make
statistical decision by optimizing a objective function. In this section several comparison theorems are provided to compare the statistical properties of two decisions
with different objective values. We first consider the situation where the decision space does not depend on $n$. An application of the corresponding results is
estimation for parameters, where the decision space is the set on which the parameters are valued. The second subsection discusses the situation where the decision space
depends on $n$, which covers the problem of variable selection. We use two subsections to state the results because there may be some confusion in notation if the
results for the first situation are viewed as special cases of those in the second situation; see Remark \ref{rk:2s}.

\subsection{When the decision space does not depend on $n$}\label{subsec:non}
\hskip\parindent \vspace{-0.6cm}

Let $\mathfrak{D}$ denote the decision space that contains all statistical decisions of interest. Suppose that we need to make inferences based on the global solution to
the optimization problem
\begin{equation}\label{sop}\min_{x\in\mathfrak{D}}\psi_n(x,\X_n),\end{equation}where the objective function $\psi_n$ is a map from
$\mathfrak{D}\times\mathcal{X}_n$ to $\mathbb{R}$. In general, the problem in \eqref{sop} is proposed because its solution can asymptotically lie in a desirable subset
$\mathfrak{A}$ of $\mathfrak{D}$ that contains all ``good" decisions. This property of the global solution can be viewed as a consistency property.

Consider the situations where the global solution to \eqref{sop} is difficult to obtain. Suppose that there are $K$ candidate solutions,
$\xi_n^{(1)},\ldots,\xi_n^{(K)}$. In practice, we always use $\xi_n^*$, which denotes the one that takes the smallest value of $\psi_n(\cdot,\,\X_n)$ among them, as the
final decision. For each $\xi_n^{(k)}$, $k=1,\ldots,K$, $\xi_n^*$ is a ``better" solution since the inequality $$\psi_n(\xi_n^*,\X_n) \leqslant
\psi_n(\xi_n^{(k)},\X_n)$$ always holds. Here we define ``BSP" as ``such a better solution is more likely to lie in $\mathfrak{A}$", and discuss whether it holds.

Let $\mathfrak{B}$ be another subset of $\mathfrak{D}$, which contains relatively bad decisions compared to $\mathfrak{A}$.

\begin{definition}\label{def:bds} We say that $\{\psi_n\}$ strongly separates
$\mathfrak{A}$ from $\mathfrak{B}$, or $\{\psi_n\}$ has the strong separation property with respect of $\mathfrak{A}$ and $\mathfrak{B}$, if as
$n\to\infty$,\begin{equation}\label{ssp} P\left(\sup_{x\in \mathfrak{A}}\psi_n(x,\X_n)<\inf_{y\in \mathfrak{B}}\psi_n(y,\X_n)\right)\to1.\end{equation}We say that
$\{\psi_n\}$ (weakly) separates $\mathfrak{A}$ from $\mathfrak{B}$, or $\{\psi_n\}$ has the (weak) separation property with respect of $\mathfrak{A}$ and $\mathfrak{B}$,
if for all $x\in \mathfrak{A},\ y\in\mathfrak{B}$,
\begin{equation}\label{wsp} \limsup_{n\to\infty}\left[\psi_n(x,\X_n)-\psi_n(y,\X_n)\right]<0\quad\text{(a.s.)},\end{equation}where ``a.s." denotes ``almost surely".
\end{definition}

It is worthwhile noting that the strong separation property needs not to imply the separation property. We use the word ``strong" to distinguish the two properties just
because the former is generally more difficult to verify and can lead to stronger results.

\begin{remark}\label{rk:scale}For convenience in asymptotic analysis, we often consider a scaled objective function. It should be pointed out that \eqref{ssp} holds if
$$P\left(\sup_{x\in \mathfrak{A}}\psi_n(x,\X_n)/a_n<\inf_{y\in \mathfrak{B}}\psi_n(y,\X_n)/a_n\right)\to1$$for arbitrary sequence of positive numbers $\{a_n\}$, and that
\eqref{wsp} holds if $$\limsup_{n\to\infty} \frac{\psi_n(x,\X_n)-\psi_n(y,\X_n)}{a_n}<0\quad\text{(a.s.)}$$
for a sequence of positive numbers $\{a_n\}$ with
$a_n^{-1}=O(1)$.
\end{remark}

Roughly speaking, the strong separation property requires that the level set corresponding to smaller objective values is asymptotically identical to the set of good
decisions. It shows the consistency of the objective function to the statistical properties of interest. We can immediately prove the following result that this property
implies BSP, where the statistical properties of interest are described with the probability of locating in the set $\mathfrak{A}$ of good decisions.

\begin{theorem}[Strong Comparison Theorem]\label{th:lt} Suppose that
$\{\psi_n\}$ strongly separates $\mathfrak{A}$ from $\mathfrak{B}$. For all $n\in\mathbb{N}$, $\xi_n$ and $\eta_n$ are statistics valued in $\mathfrak{D}$ satisfying
$P(\xi_n\in\mathfrak{A}\cup\mathfrak{B},\ \eta_n\in\mathfrak{A}\cup\mathfrak{B})\to1$ as $n\to\infty$. If $\psi_n(\xi_n,\X_n) \leqslant \psi_n(\eta_n,\X_n)$ for all $n$,
then
$$\liminf_{n\to\infty}\left[P(\xi_n\in \mathfrak{A})-P(\eta_n\in \mathfrak{A})\right]\geqslant0.$$ \end{theorem}

\begin{proof} For $\omega\in\{\eta_n\in \mathfrak{A},\ \xi_n\in \mathfrak{B}\}\subset\Omega$, if $\omega\in \left\{\sup_{x\in
\mathfrak{A}}\psi_n(x,\X_n)<\inf_{y\in\mathfrak{B}} \psi_n(y,\X_n)\right\}\cap\{\xi_n\in\mathfrak{A}\cup\mathfrak{B},\ \eta_n\in\mathfrak{A}\cup\mathfrak{B}\}$, then
\begin{eqnarray*}\psi_n\big(\eta_n(\X_n(\omega)),\X_n(\omega)\big)\leqslant \sup_{x\in \mathfrak{A}}\psi_n(x,\X_n(\omega))<\inf_{y\in \mathfrak{B}}\psi_n(y,\X_n(\omega))
\leqslant \psi_n\big(\xi_n(\X_n(\omega),\X_n(\omega)\big),\end{eqnarray*}which leads to a contradiction. Therefore, $\omega\notin \left\{\sup_{x\in
\mathfrak{A}}\psi_n(x,\X_n)<\inf_{y\in\mathfrak{B}} \psi_n(y,\X_n)\right\}\cap\{\xi_n\in\mathfrak{A}\cup\mathfrak{B},\ \eta_n\in\mathfrak{A}\cup\mathfrak{B}\}$, which
implies $P(\eta_n\in \mathfrak{A},\ \xi_n\in \mathfrak{B})\to0$. We have $P(\eta_n\in \mathfrak{A})=P(\eta_n\in\mathfrak{A},\ \xi_n\in \mathfrak{B}) +P(\eta_n\in
\mathfrak{A},\ \xi_n\in \mathfrak{A})+P(\eta_n\in\mathfrak{A},\ \xi_n\notin \mathfrak{A}\cup\mathfrak{B})=P(\xi_n\in \mathfrak{A})-P(\xi_n\in \mathfrak{A},\ \eta_n\in
\mathfrak{B})+o(1)$. This completes the proof.
\end{proof}

Recall that, for a decision $\xi_n$, the property that $P(\xi_n\in\mathfrak{A})\to1$ can be viewed as a consistency property of $\xi_n$. The following theorem shows that
the strong separation property of $\{\psi_n\}$ is often stronger than the consistency of the minimum of $\psi_n$.
\begin{theorem}\label{th:con}Suppose that $\{\psi_n\}$ strongly separates $\mathfrak{A}$ from $\mathfrak{B}$. If $\xi_n=\arg\min_{x\in\mathfrak{D}}\psi_n(x,\X_n)$ exists and
$P(\xi_n\in\mathfrak{A}\cup\mathfrak{B})\to1$ as $n\to\infty$, then $$\lim_{n\to\infty}P(\xi_n\in\mathfrak{A})\to1.$$ \end{theorem}

\begin{proof}This theorem follows from $\{\sup_{x\in \mathfrak{A}}\psi_n(x,\X_n)<\inf_{y\in
\mathfrak{B}}\psi_n(y,\X_n)\}\cap\{\xi_n\in\mathfrak{A}\cup\mathfrak{B}\}\subset\{\xi_n\in\mathfrak{A}\}$.
\end{proof}

We next discuss BSP with the separation property. This weaker property cannot directly imply BSP, and more conditions are needed.

\begin{theorem}[(Weak) Comparison Theorem]\label{th:lt2} Suppose that
$\{\psi_n\}$ separates $\mathfrak{A}$ from $\mathfrak{B}$. Denote the set of probability one where \eqref{wsp} holds by $E(x,y)$ and write $E=\cap_{x\in
\mathfrak{A},y\in \mathfrak{B}}E(x,y)$. For all $n\in\mathbb{N}$, $\xi_n$ and $\eta_n$ are statistics valued in $\mathfrak{D}$ satisfying
$P(\xi_n\in\mathfrak{A}\cup\mathfrak{B},\ \eta_n\in\mathfrak{A}\cup\mathfrak{B})\to1$ as $n\to\infty$. If $\psi_n(\xi_n,\X_n)\leqslant \psi_n(\eta_n,\X_n)$ for all $n$,
then\begin{equation}\label{wct}\liminf_{n\to\infty}\left[P(\xi_n\in \mathfrak{A})-P(\eta_n\in \mathfrak{A})\right]\geqslant-P(\Omega\setminus E).\end{equation}
\end{theorem}

\begin{proof} For $\omega\in\{\eta_n\in \mathfrak{A},\ \xi_n\in\mathfrak{B}\}\subset\Omega$, if
$\omega\in E\cap\{\xi_n\in\mathfrak{A}\cup\mathfrak{B},\ \eta_n\in\mathfrak{A}\cup\mathfrak{B}\}$, then
\begin{eqnarray*}\limsup_{n\to\infty}\left[\psi_n\big(\eta_n(\X_n(\omega)),\X_n(\omega)\big)-\psi_n\big(\xi_n(\X_n(\omega)),\X_n(\omega)\big)\right]
<0.\end{eqnarray*}This is a contradiction. Therefore, $\omega\notin E\cap\{\xi_n\in\mathfrak{A}\cup\mathfrak{B},\ \eta_n\in\mathfrak{A}\cup\mathfrak{B}\}$ for
sufficiently large $n$, which implies $P(\eta_n\in \mathfrak{A},\ \xi_n\in\mathfrak{B})\leqslant P(\Omega\setminus E)+1-P(\xi_n\in\mathfrak{A}\cup\mathfrak{B},\
\eta_n\in\mathfrak{A}\cup\mathfrak{B})$ for sufficiently large $n$. It follows that $P(\eta_n\in \mathfrak{A})= P(\eta_n\in\mathfrak{A},\ \xi_n\in
\mathfrak{B})+P(\eta_n\in \mathfrak{A},\ \xi_n\in \mathfrak{A})+P(\eta_n\in\mathfrak{A},\ \xi_n\notin \mathfrak{A}\cup\mathfrak{B})\leqslant P(\xi_n\in
\mathfrak{A})-P(\xi_n\in \mathfrak{A},\ \eta_n\in\mathfrak{B})+P(\eta_n\in \mathfrak{A},\ \xi_n\in \mathfrak{B})+1-P(\xi_n\in \mathfrak{A}\cup\mathfrak{B})\leqslant
P(\xi_n\in \mathfrak{A})+P(\Omega\setminus E)+1-P(\xi_n\in\mathfrak{A}\cup\mathfrak{B},\ \eta_n\in\mathfrak{A}\cup\mathfrak{B})+1-P(\xi_n\in
\mathfrak{A}\cup\mathfrak{B})$ for sufficiently large $n$, which completes the proof.
\end{proof}

If $P(\Omega\setminus E)$ in \eqref{wct} equals zero, then we can say BSP holds. Nevertheless, it is impossible to verify this condition in practice. A way for avoiding
this problem is to consider countable subsets, and we immediately obtain the following corollary.

\begin{corollary}\label{cor:w}Suppose that $\{\psi_n\}$ separates $\mathfrak{A}$ from $\mathfrak{B}$. For $n\in\mathbb{N}$, $\xi_n$ and $\eta_n$ are
statistics valued in a countable subset of $\mathfrak{D}$ satisfying $P(\xi_n\in\mathfrak{A}\cup\mathfrak{B},\ \eta_n\in\mathfrak{A}\cup\mathfrak{B})\to1$ as
$n\to\infty$. If $\psi_n(\xi_n,\X_n)\leqslant \psi_n(\eta_n,\X_n)$ for all $n$, then
$$\liminf_{n\to\infty}\left[P(\xi_n\in \mathfrak{A})-P(\eta_n\in \mathfrak{A})\right]\geqslant0.$$\end{corollary}

\begin{remark}\label{rk:ctb}
For many cases, $\mathfrak{D}$ is a separable set. It is usually enough to consider the decisions in its countable and dense subset in practice. For example, to estimate
a scalar parameter, we can always consider the estimators valued in the set of all rational numbers, which is countable and dense in $\mathbb{R}$. In this sense, BSP
follows from the separation property of $\{\psi_n\}$.
\end{remark}

\subsection{When the decision space depends on $n$}\label{subsec:don}
\hskip\parindent \vspace{-0.6cm}

In this subsection the decision space of interest $\mathfrak{D}_n$ depends on the sample size $n$. Suppose that we need to consider the optimization problem
\begin{equation*}\min_{x\in\mathfrak{D}_n}\psi_n(x,\X_n),\end{equation*}where the objective function $\psi_n$ is a map from
$\mathfrak{D}_n\times\mathcal{X}_n$ to $\mathbb{R}$. Different from the results in Section \ref{subsec:non}, we need to consider sequences of decisions. Denote
$\mathfrak{D}=\prod_{n=1}^\infty\mathfrak{D}_n$, and let $\mathfrak{A}$ be the subset of $\mathfrak{D}$ that contains sequences of good decisions. The statistical
property of a decision sequence we concern here is whether it lies in $\mathfrak{A}$. Let $\mathfrak{B}$ be another subset of $\mathfrak{D}$. Denote
$\mathfrak{D}^*=\mathfrak{A}\cup\mathfrak{B}$. The definition and theoretical results are parallel to those in Section \ref{subsec:non}, and the proofs are almost the
same. We therefore omit the proofs in this subsection.

\begin{definition}\label{def:bds2} We say that $\{\psi_n\}$ (weakly) separates $\mathfrak{A}$ from $\mathfrak{B}$, or $\{\psi_n\}$ has the (weak)
separation property with respect to $\mathfrak{A}$ and $\mathfrak{B}$, if for all $\{x_n\}\in \mathfrak{A},\ \{y_n\}\in\mathfrak{B}$,
\begin{equation}\limsup_{n\to\infty}\left[\psi_n(x_n,\X_n)-\psi_n(y_n,\X_n)\right]<0\quad\text{(a.s.)}.\label{wsp2}\end{equation}
Furthermore, suppose that $\mathfrak{A}$ and $\mathfrak{B}$ can be written as $\mathfrak{A}=\prod_{n=1}^\infty\mathfrak{A}_n$ and
$\mathfrak{B}=\prod_{n=1}^\infty\mathfrak{B}_n$, where $\mathfrak{A}_n$ and $\mathfrak{B}_n$ are subsets of $\mathfrak{D}_n$. We say that $\{\psi_n\}$ strongly separates
$\mathfrak{A}$ from $\mathfrak{B}$, or $\{\psi_n\}$ has the strong separation property with respect to $\mathfrak{A}$ and $\mathfrak{B}$, if as $n\to\infty$,
\begin{equation*} P\left(\sup_{x\in \mathfrak{A}_n}\psi_n(x,\X_n)<\inf_{y\in
\mathfrak{B}_n}\psi_n(y,\X_n)\right)\to1.\end{equation*}
\end{definition}

\begin{theorem}[Strong Comparison Theorem]\label{th:lts} Suppose that $\mathfrak{A}$ and $\mathfrak{B}$ can be written as
$\mathfrak{A}=\prod_{n=1}^\infty\mathfrak{A}_n$ and $\mathfrak{B}=\prod_{n=1}^\infty\mathfrak{B}_n$, and $\{\psi_n\}$ strongly separates $\mathfrak{A}$ from
$\mathfrak{B}$. For two sequences of statistics $\{\xi_n\}$ and $\{\eta_n\}$ valued in $\mathfrak{D}^*$, if $\psi_n(\xi_n,\X_n)\leqslant \psi_n(\eta_n,\X_n)$ for all
$n$, then
$$\liminf_{n\to\infty}\left[P(\xi_n\in \mathfrak{A}_n)-P(\eta_n\in \mathfrak{A}_n)\right]\geqslant0$$ and
$$P\big(\{\xi_n\}\in \mathfrak{A}\big)\geqslant P\big(\{\eta_n\}\in \mathfrak{A}\big).$$\end{theorem}

\begin{theorem}\label{th:con2}Under the conditions in Theorem \ref{th:lts},
if $\xi_n=\arg\min_{x\in\mathfrak{D}^*}\psi_n(x,\X_n)$ exists, then
$$\lim_{n\to\infty}P(\xi_n\in\mathfrak{A}_n)\to1.$$ \end{theorem}

\begin{theorem}[(Weak) Comparison Theorem]\label{th:lt2} Suppose that $\{\psi_n\}$ (weakly) separates $\mathfrak{A}$ from $\mathfrak{B}$.
Denote the set of probability one where \eqref{wsp2} holds by $E(\{x_n\},\{y_n\})$ and write $E=\cap_{\{x_n\}\in \mathfrak{A},\{y_n\}\in
\mathfrak{B}}E(\{x_n\},\{y_n\})$. For two sequences of statistics $\{\xi_n\}$ and $\{\eta_n\}$ valued in $\mathfrak{D}^*$, if $\psi_n(\xi_n,\X_n)\leqslant
\psi_n(\eta_n,\X_n)$ for all $n$, then $$P\big(\{\xi_n\}\in \mathfrak{A}\big)-P\big(\{\eta_n\}\in \mathfrak{A}\big)\geqslant-P(\Omega\setminus E).$$ Furthermore, if
$\mathfrak{A}$ and $\mathfrak{B}$ can be written as $\mathfrak{A}=\prod_{n=1}^\infty\mathfrak{A}_n$ and $\mathfrak{B}=\prod_{n=1}^\infty\mathfrak{B}_n$, then for
sufficiently large $n$,
$$P(\xi_n\in \mathfrak{A}_n)-P(\eta_n\in \mathfrak{A}_n)\geqslant-P(\Omega\setminus E).$$
\end{theorem}

\begin{corollary}\label{cor:w}Suppose that $\{\psi_n\}$ weakly separates $\mathfrak{A}$ from $\mathfrak{B}$.
For two sequences of statistics $\{\xi_n\}$ and $\{\eta_n\}$ valued in a countable subset of $\mathfrak{D}^*$, if $\psi_n(\xi_n,\X_n)\leqslant \psi_n(\eta_n,\X_n)$ for
all $n$, then $$P\big(\{\xi_n\}\in \mathfrak{A}\big)\geqslant P\big(\{\eta_n\}\in \mathfrak{A}\big).$$ Furthermore, if $\mathfrak{A}$ and $\mathfrak{B}$ can be written
as $\mathfrak{A}=\prod_{n=1}^\infty\mathfrak{A}_n$ and $\mathfrak{B}=\prod_{n=1}^\infty\mathfrak{B}_n$, then for sufficiently large $n$,
$$P(\xi_n\in \mathfrak{A}_n)\geqslant P(\eta_n\in \mathfrak{A}_n).$$\end{corollary}

\begin{remark}\label{rk:ctb2}
In practice, we only have the observed values of $\xi_n$ and $\eta_n$ for a specified $n$. It can be assumed that they are respectively from two sequences $\{\xi_n\}$
and $\{\eta_n\}$ valued in a countable subset of $\mathfrak{D}^*$, especially when the decision space $\mathfrak{D}_n$ is a finite set for each $n$. In this sense, as in
Remark \ref{rk:ctb}, the separation property of $\{\psi_n\}$ is sufficient to imply BSP.
\end{remark}

\begin{remark}\label{rk:2s}It can be seen that there may be some confusion in notation if the situation where the decision
space does not depend on $n$ is viewed as a special case of the situation where the decision space depends on $n$. For example, if we consider the estimation problem of
$\theta\in\Theta$ and handle it as a special case of the second situation, then we need to view a decision $\theta\in\Theta$ as a sequence
$(\theta,\ldots,\theta,\ldots,)$, and thus the decision space $\mathfrak{D}=\{(\theta,\ldots,\theta,\ldots,):\theta\in\Theta\}$, which has not the natural form
$\mathfrak{D}=\prod_{n=1}^\infty\Theta$. The form of $\mathfrak{A}$ is also strange.\end{remark}

In the rest of this paper, we omit the sample $\X_n$ in $\psi_n(\cdot,\X_n)$ and write $\psi_n(\cdot)$ for emphasizing the decision variable.

\section{Greater likelihood principle}\label{sec:smle}
\hskip\parindent \vspace{-0.8cm}

We have shown in Section \ref{sec:th} that the separation properties of an objective function can imply the corresponding BSP. Despite simplicity, these results are
effective to establish BSP since many objective functions indeed possess the separation properties under reasonable conditions. From this section to Section
\ref{sec:bsr}, we show that BSP holds for several important statistical optimization problems by use of them. This section discusses the problem associated with maximum
likelihood estimation.

\subsection{Separation properties of the likelihood function}\label{subsec:smleres}
\hskip\parindent \vspace{-0.6cm}

Let the data $X_1,\ldots,X_n$ be i.i.d. from a probability density function (p.d.f.) $f(\cdot,\theta)$ with respect to a $\sigma$-finite measure $\mu$ on $\mathbb{R}^p$,
where $\theta$ lies in the parameter space $\Theta\subset\mathbb{R}^q$. The likelihood function is $$l_n(\theta)=\prod_{i=1}^nf(X_i,\theta),$$and the maximum likelihood
estimator (MLE) is the solution to the optimization problem \begin{equation}\label{mle}\max_{\theta\in\Theta}\,l_n(\theta).\end{equation}For convenience, we write
\eqref{mle} as the problem of minimizing the negative
log-likelihood\begin{equation}\label{lmle}\min_{\theta\in\Theta}\big[-\log\big(l_n(\theta)\big)\big].\end{equation}The MLE is commonly used due to its well-known high
asymptotic efficiency. However, when the negative log-likelihood has multiple local minima, the MLE is difficult to compute (Gan and Jiang 1999).

When estimation accuracy is concerned, it is common to use the probability of lying in a neighborhood of the true parameter to evaluate an estimator. For a consistent
estimator, this probability converges to one as the sample size goes to infinity. Following this way, we define ``good" decisions in discussing BSP for the MLE problem,
and show that, for two estimators, the better one with greater likelihood has larger probability of lying in a sufficiently small neighborhood of $\theta_0$ under
regularity conditions, where $\theta_0$ denotes the true parameter. This result, called \emph{greater likelihood principle} in this paper, is a special case of BSP and
can be viewed as a supplementary of the maximum likelihood principle.

By the results in Section \ref{sec:th}, we can establish BSP via the separation properties of the objective function. Some assumptions and lemmas are needed.

Denote
\begin{equation}\label{ieq}s(\theta,\theta_0)=-\int\log\big(f(x,\theta)\big)f(x,\theta_0)d\mu(x).\end{equation}

\begin{assumption}\label{as:ud} For all $\theta_1,\theta_2\in\Theta$, $f(\cdot,\theta_1)=f(\cdot,\theta_2)$ (a.s.) implies $\theta_1=\theta_2$.\end{assumption}

\begin{assumption}\label{as:g} For all $\theta\in\Theta$, $\int|\log\big(f(x,\theta)\big)|f(x,\theta)d\mu(x)<\infty$.\end{assumption}

\begin{assumption}\label{as:g2} For all $\theta_0\in\Theta$, $s(\cdot,\theta_0)$ is continuous on $\Theta$ and $\liminf_{x\to b}s(x,\theta_0)>s(\theta_0,\theta_0)$
for all $b\in \mathcal{C}^*(\Theta)\setminus\Theta$, where $\mathcal{C}^*(\Theta)=\mathcal{C}(\Theta)$ if $\Theta$ is bounded and
$\mathcal{C}^*(\Theta)=\mathcal{C}(\Theta)\cup\{\infty\}$ otherwise. Here $\mathcal{C}(\Theta)$ denotes the closure of $\Theta$.\end{assumption}

\begin{lemma}\label{lemma:fmin} Let $h$ be a continuous function defined in $D\subset\mathbb{R}^q$. Suppose that $h$ has a unique minimum $x_0$, i.e., for all
$x\neq x_0$, $h(x)>h(x_0)$. Furthermore, for all $b\in \mathcal{C}^*(D)\setminus D$, $\liminf_{x\to b}h(x)>h(x_0)$. Then for all $\epsilon>0$, there exists $\delta>0$
such that $\{x\in D:\ h(x)-h(x_0)\leqslant\delta\}\subset B(x_0,\epsilon)$, where $B(x_0,\epsilon)=\{x\in\mathbb{R}^q:\ \|x-x_0\|\leqslant\epsilon\}$.\end{lemma}

\begin{proof} For any sequence of positive numbers $\{a_n\}$ with $a_n\to0$ as $n\to\infty$, assume that there exist $x_n\in D$ and $\epsilon_0>0$ such that
$h(x_n)-h(x_0)\leqslant a_n$ but $|x_n-x_0|>\epsilon_0$. Therefore $h(x_n)\to h(x_0)$. Since any limit point of $\{x_n\}$ in $\mathcal{C}^*(D)$ cannot be $x_0$, this is
in contradiction to the condition that $x_0$ is the unique minimum of $h$.  \end{proof}

\begin{lemma}\label{lemma:fg} If Assumptions \ref{as:ud} and \ref{as:g} hold, then for all $\theta_0\in\Theta$, $s(\cdot,\theta_0)$ in \eqref{ieq},
as a function defined on $\Theta$, attains its minimum uniquely at $\theta_0$.
\end{lemma}

The above lemma and its proof can be found in many places; see, e.g., Wald (1949) and Van der Vaart (1998).

Under Assumptions \ref{as:ud}--\ref{as:g2}, by Lemmas \ref{lemma:fmin} and \ref{lemma:fg}, for all $\epsilon>0$, there exists $\delta(\epsilon)>0$ such that
$\{\theta\in\Theta:\ s(\theta,\theta_0)-s(\theta_0,\theta_0)\leqslant\delta(\epsilon)\}\subset B(\theta_0,\epsilon)$. Denote $B_s(\theta_0,\epsilon)=\{\theta\in\Theta:\
s(\theta,\theta_0)-s(\theta_0,\theta_0)\leqslant\delta(\epsilon)\}$ and consider\begin{equation}\label{ab}\mathfrak{A}^\epsilon=B_s(\theta_0,\epsilon),\
\mathfrak{B}^\epsilon=\Theta\setminus B_s(\theta_0,\epsilon).\end{equation} Note that for all $\theta\in\Theta$,
$$\frac{-\log(l_n(\theta))}{n}=-\frac{\log\big(f(X_1,\theta)\big)+\cdots+\log\big(f(X_n,\theta)\big)}{n}\to s(\theta,\theta_0)\quad \text{(a.s.)}.$$
We can immediately obtain the
following theorem by Definition \ref{def:bds} and Remark \ref{rk:scale}.

\begin{theorem} \label{th:bsi} Under Assumptions \ref{as:ud}--\ref{as:g2}, for all $\epsilon>0$, $\{-\log(l_n)\}$ separates
$\mathfrak{A}^\epsilon$ from $\mathfrak{B}^\epsilon$ in \eqref{ab}. \end{theorem}

\begin{remark}The conditions for the separation property of $\{-\log(l_n)\}$ are weaker than those for the consistency of MLE in Wald (1949). Furthermore,
Our results in this section neither rely on the existence of an MLE nor require that $\Theta$ is an open subset.\end{remark}

 Although the
separation property is sufficient for BSP in practical use by Remark \ref{rk:ctb}, the strong separation property is still of theoretical interest. We next discuss it
for the likelihood function. Some stronger conditions are needed.

\begin{assumption}\label{as:cs} The family $\{f(\cdot,\theta)\}_{\theta\in\Theta}$ has a common support set $\s{S}=\{x\in\mathbb{R}^p:\ 0<f(x,\theta)<\infty\}$.
For all $x\in\s{S}$, $f(x,\cdot)$ is continuous on $\Theta$.\end{assumption}

\begin{assumption}\label{as:cg} For any $\theta\in\Theta$ and any compact subset ${K}$ of $\Theta$, $$\int\sup_{\phi\in{K}}\big|\log\big(f(x,\phi)\big)\big|
f(x,\theta)d\mu(x)<\infty.$$\end{assumption}

Take $\mathfrak{A}^\epsilon$ as in \eqref{ab}. Instead of $\mathfrak{B}^\epsilon$ in \eqref{ab}, take $\mathfrak{B}_*^\epsilon$ as any compact subset of $\Theta\setminus
B_s(\theta_0,\epsilon)$.

\begin{theorem} \label{th:sbsi} Under Assumptions \ref{as:ud}, \ref{as:g2}, \ref{as:cs}, and \ref{as:cg}, for all $\epsilon>0$, $\{-\log(l_n)\}$
strongly separates $\mathfrak{A}^\epsilon$ from $\mathfrak{B}_*^\epsilon$. \end{theorem}

\begin{proof} Consider the Banach space of all continuous function on $B_s(\theta_0,\epsilon)$,
which is separable since $B(\theta_0,\epsilon)$ is a compact subset of $\mathbb{R}^q$. By Assumption \ref{as:cs},
$-\log\big(f(X_1,\theta)\big),\ldots,-\log\big(f(X_n,\theta)\big)$ are i.i.d. random variables valued in this Banach space. By Assumption \ref{as:cg} and the law of
large numbers in Banach spaces (see, e.g., Corollary 7.10 in Ledoux and Talagrand 1980),
$$\sup_{\theta\in \mathfrak{A}^\epsilon}\big|\big[-\log(l_n(\theta))\big]- s(\theta,\theta_0)\big|\to0\quad\text{(a.s.)},$$ which implies
\begin{equation}\label{k1}\sup_{\theta\in
\mathfrak{A}^\epsilon}\big[-\log(l_n(\theta))\big]\to \sup_{\theta\in \mathfrak{A}^\epsilon}s(\theta,\theta_0)\quad\text{(a.s.)}.\end{equation}Similarly, we
have\begin{equation}\label{k2}\inf_{\theta\in \mathfrak{B}_*^\epsilon}\big[-\log(l_n(\theta))\big]\to \inf_{\theta\in
\mathfrak{B}_*^\epsilon}s(\theta,\theta_0)\quad\text{(a.s.)}.\end{equation}Since $\mathfrak{B}_*^\epsilon$ is compact, there exists $\delta_1>0$ such that
$s(\theta,\theta_0)\geqslant s(\theta_0,\theta_0)+\delta+\delta_1$ for all $\theta\in\mathfrak{B}_*^\epsilon$. Consequently, by \eqref{k1} and \eqref{k2},
$$P\left(\sup_{\theta\in \mathfrak{A}^\epsilon}\big[-\log(l_n(\theta))\big]<\inf_{\theta\in \mathfrak{B}_*^\epsilon}\big[-\log(l_n(\theta))\big]\right)\to1,$$
which completes the
proof.\end{proof}

\begin{remark}If Assumptions \ref{as:cs} and \ref{as:cg} hold, it can be proved that $s(\cdot,\theta_0)$ is continuous on $\Theta$,
which is assumed in Assumption \ref{as:g2}. \end{remark}

By the two Strong Comparison Theorems, Theorems \ref{th:lt} and \ref{th:lts}, the strong separation property of the objective function provides a more strict guarantee
of BSP than its weak analogue. However, at a price of this strictness, more restrictive conditions are required for verifying the strong separation property. By Theorem
\ref{th:lt}, for comparing two estimators $\xi_n$ and $\eta_n$ via the strong separation property stated in Theorem \ref{th:sbsi}, we require
$P(\xi_n\in\mathfrak{A}^\epsilon\cup\mathfrak{B}_*^\epsilon)\to1$ and $P(\eta_n\in\mathfrak{A}^\epsilon\cup\mathfrak{B}_*^\epsilon)\to1$.

\subsection{A simulation study}\label{subsec:smlesimu}
\hskip\parindent \vspace{-0.6cm}

In this subsection we conduct a small simulation study to verify the greater likelihood principle in finite-sample cases. Consider a location family with the density
function
$$f(x,\theta)=f_0(x-\theta),$$where $\theta\in{\mathbb{R}}$ is the unknown parameter we want to estimate based on the i.i.d. observations $X_1,\ldots,X_n$. Three types
of $f_0$ are used: the standard normal distribution, $t$ distribution with 5 degrees of freedom, and the Cauchy distribution with density
$f_0(x)=\big[\pi(x^2+1)\big]^{-1}$. It is known that the likelihood functions for the latter two cases often have multiple maximum. We compare three simple methods, the
sample median, the trimmed mean removing $50\%$ extreme values, and the method that selects the better one of the two estimators with greater likelihood as the final
estimator. Given the true parameter $\theta_0=0$, we repeat 10,000 times to compute mean squares errors (MSEs) of the three estimators for various sample sizes, and the
results are displayed in Table \ref{tab:cau}. It can be seen that the results follow the greater likelihood principle well: the ``better" solution always yields the
smallest MSEs among the three estimators.

\begin{table} \caption{\label{tab:cau}MSE comparisons in Section \ref{subsec:smlesimu}} \centering\vspace{2mm}
{\begin{tabular}{*{10}{lllccccccc}}\hline&\quad\quad&&\quad\quad&\multicolumn{6}{c}{$n$}\\\cline{5-10}&&&&$10$&$15$&$20$&$25$&$30$&$35$\\\hline \multirow{3}{*}{Normal}
&&     median &&0.1361&0.1019&0.0728&0.0623&0.0502&0.0447\\
&&trimmed mean&&0.1113&0.0798&0.0588&0.0472&0.0393&0.0343\\
&&better      &&0.1093&0.0776&0.0574&0.0459&0.0382&0.0333\\\hline \multirow{3}{*}{$t_5$}
&&     median &&0.1588&0.1159&0.0824&0.0701&0.0568&0.0508\\
&&trimmed mean&&0.1393&0.0961&0.0698&0.0559&0.0465&0.0409\\
&&better      &&0.1383&0.0951&0.0694&0.0555&0.0463&0.0405\\\hline \multirow{3}{*}{Cauchy}
&&     median &&0.3360&0.2056&0.1427&0.1109&0.0905&0.0804\\
&&trimmed mean&&0.4929&0.2236&0.1628&0.1221&0.1027&0.0827\\
&&better      &&0.3260&0.1857&0.1333&0.1001&0.0845&0.0720\\\hline
\end{tabular}}
\end{table}

\section{Better subsample selection under contaminated models}\label{sec:boe}
\hskip\parindent \vspace{-0.8cm}

Let $\s{F}^p$ denote the set of all cumulative distribution functions (c.d.f.) on $\mathbb{R}^p$. Suppose that we are interested in making inferences for the unknown
parameter $\theta$ of a parametric family $\{F_\theta\}_{\theta\in\Theta}$ based on i.i.d. observations $X_1,\ldots,X_n$, where $F_\theta\in\s{F}^p$ for all $\theta$ and
the parameter space $\Theta$ is a subset of $\mathbb{R}^q$. When the observations include some outliers, a commonly used assumption for describing this situation is that
the dataset is a randomly mixed batch of $n$ ``good" observations and outliers, and that each single observation with probability $1-\epsilon$ is a ``good" one, with
probability $\epsilon$ an outlier, where $\epsilon\in[0,1/2]$ (Huber 1981). Under this assumption, the observations are drawn from the contaminated population, i.e.,
\begin{equation}\label{cm}X_1,\ldots,X_n\ \text{i.i.d.}\ \sim (1-\epsilon)F_\theta+\epsilon G,\end{equation} where $G\in\s{F}^p$ is the contamination distribution. Here we
consider a simplified model by removing the randomness of $X_i$ being a good observation or an outlier. Denote the set of all subsequences of $\{n\}_{n=1,2,\ldots}$ by
$\mathfrak{S}$, i.e., \begin{equation*} \mathfrak{S}=\big\{\{k_n\}:\ k_n\in\mathbb{N},\ k_1<k_2<\cdots\big\}.\end{equation*} Take a nondecreasing integer sequence
$\{l_n\}$ satisfying $l_n/n\to1-\epsilon$ as $n\to\infty$. For $\{k_n^0\}\in\mathfrak{S}$, let
\begin{equation}\label{a0}\s{A}_{0n}=\{k_{1}^0,\ldots,k_{l_n}^0\}\end{equation} denote the index set of all good observations. Assume that
\begin{equation}\label{cm2}X_1,\ldots,X_n\ \text{are independently drawn as}\ X_i\sim F_\theta\ \text{for}\ i\in\s{A}_{0n}\ \text{and}\ X_i\sim G\
\text{for}\ i\notin\s{A}_{0n}.\end{equation}The model \eqref{cm2} is asymptotically equivalent to \eqref{cm} in the sense that the two empirical distributions based on
the observations generated from both of them have the identical limit $(1-\epsilon)F_\theta+\epsilon G$ as $n\to\infty$.

\begin{remark}\label{rk:se} The assumption that $\s{A}_{0n}$ is a segment of a subsequence is technical. Under this assumption,
the observations can be viewed as a sequence of random variables, and thus some limit theory on sequences of random variables can be applied such as the strong law of
large numbers. Otherwise, we may have to consider the observations as triangle arrays, and more restrictive conditions are required to establish the corresponding
asymptotic results. On the practical aspect, this assumption is also reasonable.\end{remark}

The set $\s{A}_{0n}$ in \eqref{cm2} consists of the indices of all good observations. An ideal method for robust inferences is based on all good observations, i.e., we
first correctly identify $\s{A}_{0n}$. We refer to the method of selecting $\s{A}_{0n}$ by optimizing some criteria as \emph{best subsample selection}, parallel to best
subset regression in variable selection. The minimum covariance determinant estimate (Rousseeuw 1985) and least trimmed squares estimate (Rousseeuw 1984) are instances
of best subsample selection-based estimates. In general, it is impossible to exactly select $\s{A}_{0n}$ itself since $\epsilon$ is usually unknown. A practical purpose
is to select a subset of $\s{A}_{0n}$.

The estimates based on best subsample selection have high breakdown values (Hubert, Rousseeuw, and Van Aelst 2008), whereas their asymptotic properties are difficult to
derive. Limited results were obtained under uncontaminated models, i.e., $\epsilon=0$ in \eqref{cm}; see, e.g., Rousseeuw and Leroy (1987), Butler, Davies, and Jhun
(1993), and Agull\'{o}, Croux, and Van Aelst (2008). To the best of the author's knowledge, there is no work on the asymptotics of best subsample selection or related
estimates under contaminated models such as \eqref{cm} or \eqref{cm2}. In this section we discuss whether BSP for best subsample selection (asymptotically) holds under
model \eqref{cm2}.

The statistical optimization problem in best subsample selection can be formulated as follows. For all $n$, the decision space is
\begin{equation}\label{d}\mathfrak{D}_n=\{\s{A}\subset \mathbb{Z}_n:\ |\s{A}|=m\},\end{equation}
where $m=m_n<n$ is a pre-specified integer and $|\cdot|$ denotes cardinality. The best subsample of size $m$ is the solution
to\begin{equation}\label{bssp}\min_{\s{A}\in\mathfrak{D}_n}\psi_n(\s{A}),\end{equation} where $\psi_n$ is the objective function. Two types of objective functions will
be discussed in Sections \ref{subsec:smle} and \ref{subsec:smd}, respectively. It is often difficult to attain the global solution to \eqref{bssp}. In this section we
prove the separation property of the two types of objective functions. This property implies that better subsamples are more likely to be subsets of $\s{A}_{0n}$ by the
comparison theorems.

Denote
\begin{eqnarray}\begin{array}{l}\mathfrak{S}_0=\big\{\{k_n\}\in\mathfrak{S}:\ \{k_1,\ldots,k_m\}\subset\s{A}_{0n}\ \text{for all}\ n\big\},
\\\mathfrak{S}_1=\big\{\{k_n\}\in\mathfrak{S}:\ \sum_{i=1}^mI(k_i\notin\s{A}_{0n})/m\to\alpha>0\ \text{as}\ n\to\infty\big\},\end{array}\label{rob0}
\end{eqnarray}and
\begin{eqnarray}\begin{array}{l}\mathfrak{A}=\big\{\{\s{A}_n\}:\ \s{A}_n=\{k_1,\ldots,k_m\}\ \text{for all}\ n,\ \{k_n\}\in\mathfrak{S}_0\big\}_{n\in\mathbb{N}},
\\\mathfrak{B}=\big\{\{\s{A}_n\}:\ \s{A}_n=\{k_1,\ldots,k_m\}\ \text{for all}\ n,\ \{k_n\}\in\mathfrak{S}_1\big\}_{n\in\mathbb{N}},\end{array}\label{robab}
\end{eqnarray}where $I$ is the indicator function. In this section $\mathfrak{A}$ serves as the space of good decisions. In fact, the asymptotic
results in this section also hold if $\mathfrak{S}_0$ in \eqref{rob0} is replaced by $\mathfrak{S}_0=\big\{\{k_n\}\in\mathfrak{S}:\
\sum_{i=1}^mI(k_i\notin\s{A}_{0n})/m\to0\ \text{as}\ n\to\infty\big\}.$

\subsection{Selection by maximum likelihood}\label{subsec:smle}
\hskip\parindent \vspace{-0.6cm}

Suppose that $F_\theta$ and $G$ in \eqref{cm2} respectively have the p.d.f.'s, $f(\cdot,\theta)$ and $g(\cdot)$, with respect to a $\sigma$-finite measure on
$\mathbb{R}^p$, where $\theta\in\Theta\subset\mathbb{R}^q$. A natural idea is to select the best subsample by maximum likelihood, i.e., the objective function in
\eqref{bssp} is taken as
\begin{equation}\label{smle}\psi_n(\s{A})=\inf_{\theta\in\Theta}\left[-\sum_{i\in\s{A}}\log \big(f(X_i,\theta)\big)\right].\end{equation}
If
\begin{equation}\label{thetah}\hat{\theta}_{\s{A}}=\arg\min_{\theta\in\Theta}\left[-\sum_{i\in\s{A}}\log \big(f(X_i,\theta)\big)\right]\end{equation}
exists for all $\s{A}\in\mathfrak{D}_n$, then we can write
\begin{equation}\label{smle0}\psi_n(\s{A})=-\sum_{i\in\s{A}}\log \big(f(X_i,\hat{\theta}_{\s{A}})\big).\end{equation}
The minimum covariance determinant method (Rousseeuw 1985), which looks for the observations whose classical covariance matrix has the lowest possible determinant, can
be viewed as an instance of the method of minimizing \eqref{smle0} if the underlying model is assumed to be a multivariate normal distribution.

We need several assumptions to establish the separation property of $\psi_n$. Denote $$s_g(\theta)=-\int\log\big(f(x,\theta)\big)g(x)d\mu(x).$$
\begin{assumption}\label{as:m} For sufficiently large $n$, $m\leqslant|\s{A}_{0n}|$, and as $n\to\infty$, $m/n\to\tau\in[1/2,1-\epsilon]$.\end{assumption}

\begin{assumption}\label{as:g2s} For all $\alpha\in[0,\epsilon/(1-\tau)]$, $\arg\min_{\theta\in\Theta}(1-\alpha)s(\theta,\theta_0)+\alpha s_g(\theta)$ exists,
where $s$ is defined in \eqref{ieq}.
\end{assumption}
Denote $\theta^*=\arg\min_{\theta\in\Theta}(1-\alpha)s(\theta,\theta_0)+\alpha s_g(\theta)$, and \begin{eqnarray*}&&\varphi(x,r)=\sup_{\theta\in\Theta\setminus
B(\theta^*,r)}f(x,\theta),
\\&&\varphi^*(x,r)=\left\{\begin{array}{ll}1,\quad\text{if}\ \varphi(x,r)\leqslant1,
\\\varphi(x,r),\ \text{otherwise}.\end{array}\right.\end{eqnarray*}

\begin{assumption}\label{as:css} The family $\{f(\cdot,\theta)\}_{\theta\in\Theta}$ has a common support set $\s{S}=\{x\in\mathbb{R}^p:\ 0<f(x,\theta)<\infty\}$.
For all $x\in\s{S}$, $f(x,\cdot)$ is continuous on $\Theta$, and $\lim_{\|\theta\|\to\infty}f(x,\theta)=0$.\end{assumption}

\begin{assumption}\label{as:ib} There exists $r^*>0$ such that $\int\log\big(\varphi^*(x,r^*)\big)f(x,\theta_0)d\mu(x)<\infty$ and
$\int\log\big(\varphi^*(x,r^*)\big)g(x)d\mu(x)<\infty$.\end{assumption}

\begin{assumption}\label{as:cfg} For all compact subset ${K}$ of $\Theta$, $$\int\sup_{\phi\in{K}}\big|\log\big(f(x,\phi)\big)\big|
f(x,\theta_0)d\mu(x)<\infty,\quad\int\sup_{\phi\in{K}}\big|\log\big(f(x,\phi)\big)\big| g(x)d\mu(x)<\infty.$$\end{assumption}


\begin{assumption}\label{as:ss} For all $\theta\in\Theta$ and $\alpha\in(0,\epsilon/(1-\tau)]$, $(1-\alpha)s(\theta,\theta_0)+\alpha
s_g(\theta)>s(\theta_0,\theta_0)$.\end{assumption}

\begin{lemma}\label{lemma:einf} Under Assumptions \ref{as:css} and \ref{as:ib}, we have
\begin{eqnarray}&&\lim_{r\to\infty}\int\log\big(\varphi(x,r_0)\big)f(x,\theta_0)d\mu(x)=-\infty,\label{lim1}
\\&&\lim_{r\to\infty}\int\log\big(\varphi(x,r_0)\big)g(x)d\mu(x)=-\infty.\label{lim2}\end{eqnarray}
\end{lemma}

\begin{proof} See Wald (1949) for the proof of \eqref{lim1}. and that of \eqref{lim2} is almost the same. \end{proof}

\begin{theorem} \label{th:bsj} Under Assumptions \ref{as:m}--\ref{as:ss}, $\{\psi_n\}$ in \eqref{smle} separates $\mathfrak{A}$ from $\mathfrak{B}$ in \eqref{robab}.
\end{theorem}

\begin{proof} For $\{\s{A}_n\}\in\mathfrak{A}$, \begin{eqnarray}&&\frac{\psi_n(\s{A}_n)}{m}
=\inf_{\theta\in\Theta}\left[-\frac{1}{m}\sum_{i\in\s{A}_n}\log \big(f(X_i,\theta)\big)\right] \nonumber\\&&\leqslant-\frac{1}{m}\sum_{i\in\s{A}_n}\log
\big(f(X_i,\theta_0)\big)\to s(\theta_0,\theta_0)\quad\text{(a.s.)}.\label{dcc}\end{eqnarray}

Consider $\{\s{A}_n\}\in\mathfrak{B}$. By Lemma \ref{lemma:einf}, there exists $r_0$ such that
\begin{eqnarray*}&&\int\log\big(\varphi(x,r_0)\big)f(x,\theta_0)d\mu(x)<-s(\theta^*,\theta_0),\\&&
\int\log\big(\varphi(x,r_0)\big)g(x)d\mu(x)<-s_g(\theta^*).\end{eqnarray*} By the strong law of large numbers, \begin{eqnarray*}-\frac{1}{m}\sum_{i\in\s{A}_n}\log
\big(\varphi(X_i,r_0)\big)&\to&-(1-\alpha)\int\log\big(\varphi(x,r_0)\big)f(x,\theta_0)d\mu(x)\\&& -\alpha \int\log\big(\varphi(x,r_0)\big)g(x)d\mu(x)\quad
\text{(a.s.)},\\-\frac{1}{m}\sum_{i\in\s{A}_n}\log \big(f(X_i,{\theta}^*)\big)&\to&(1-\alpha)s(\theta^*,\theta_0)+\alpha s_g(\theta^*)\quad
\text{(a.s.)},\end{eqnarray*}which implies
\begin{equation}\label{lmif}\liminf_{n\to\infty}\left[-\frac{1}{m}\sum_{i\in\s{A}_n}\log
\big(\varphi(X_i,r_0)\big)+\frac{1}{m}\sum_{i\in\s{A}_n}\log \big(f(X_i,{\theta}^*)\big)\right]>0\quad \text{(a.s.)}.\end{equation}Note that
\begin{eqnarray*}&&\inf_{\theta\in\Theta\setminus B(\theta^*,r_0)}\left[-\frac{1}{m}\sum_{i\in\s{A}_n}\log
\big(f(X_i,{\theta})\big)\right]\geqslant-\frac{1}{m}\sum_{i\in\s{A}_n}\log \left(\sup_{\theta\in\Theta\setminus B(\theta^*,r_0)}f(X_i,{\theta})\right)
\\&&=-\frac{1}{m}\sum_{i\in\s{A}_n}\log
\big(\varphi(X_i,r_0)\big).\end{eqnarray*}By \eqref{lmif}, for sufficiently large $n$, $$\inf_{\theta\in\Theta\setminus
B(\theta^*,r_0)}\left[-\frac{1}{m}\sum_{i\in\s{A}_n}\log \big(f(X_i,{\theta})\big)\right]>-\frac{1}{m}\sum_{i\in\s{A}_n}\log \big(f(X_i,{\theta}^*)\big)\quad
\text{(a.s.)}.$$Hence, by the law of large numbers in Banach spaces,
\begin{eqnarray}&&\frac{\psi_n(\s{A}_n)}{m}=\inf_{\theta\in\Theta}\left[-\frac{1}{m}\sum_{i\in\s{A}_n}\log
\big(f(X_i,{\theta})\big)\right]=\min_{\theta\in\s{C}(\Theta)\cap B(\theta^*,r_0)}\left[-\frac{1}{m}\sum_{i\in\s{A}_n}\log
\big(f(X_i,{\theta})\big)\right]\nonumber\\&&\to(1-\alpha)s(\theta^*,\theta_0)+\alpha s_g(\theta^*)\quad\text{(a.s.)}.\label{dcc2}\end{eqnarray} Combining \eqref{dcc}
and \eqref{dcc2}, by Assumption \ref{as:ss}, we complete the proof.
\end{proof}

If $\hat{\theta}_{\s{A}}$ in \eqref{thetah} exists for all $\s{A}\in\mathfrak{D}_n$, we can show that $\hat{\theta}_{\s{A}_n}\to\theta_0$ and $\psi_n(\s{A}_n)/m\to
s(\theta_0,\theta_0)$ (a.s.) for $\{\s{A}_n\}\in\mathfrak{A}$ under regularity conditions. From the above proof, Assumption \ref{as:ss} is actually a necessary condition
for the separation property of $\{\psi_n\}$. This assumption is generally strong. Consider a simple case of $\epsilon=\tau=1/2$, where Assumption \ref{as:ss} reduces to
\begin{equation}\label{12}s_g(\theta)>s(\theta_0,\theta_0)\quad\text{for all}\ \theta\in\Theta.\end{equation}
For $f(x,\theta)=(2\pi)^{-1/2}\exp\big(-(x-\theta)^2/2\big)$ with $\theta_0=0$, $s(\theta_0,\theta_0)=\log(2\pi)/2+1/2$, and
$s_g(\theta)=\log(2\pi)/2+1/2\int(x-\theta)^2g(x)dx$. Therefore, \eqref{12} holds if and only if
$$\int(x-\theta)^2g(x)dx>1\quad\text{for all}\ \theta\in\mathbb{R},$$which is equivalent to $\mathrm{Var}(Z)>1$, where $Z\sim g$. If $g$ is the p.d.f. of
$N(\mu,\sigma^2)$ with $\sigma\leqslant1$, then $\{\psi_n\}$ in \eqref{smle} cannot separate $\mathfrak{A}$ from $\mathfrak{B}$ no matter how far away $\mu$ is from
$\theta_0$. This example indicates that the selection by maximum likelihood may perform poorly when there are clustered outliers. In the next subsection we will provide
another subsample selection method that still works well for this case.

\subsection{Selection by minimum distance}\label{subsec:smd}
\hskip\parindent \vspace{-0.6cm}

An important class of robust estimators is the minimum distance estimator (Wolfowitz 1957), which is derived by minimizing a certain ``distance" between the observations
and the assumed population. This estimator usually possesses good robust properties, and has been discussed actively in the literature; see Donoho and Liu (1994),
Lindsay (1994), and Wu, Karunamuni, and Zhang (2012), among others. Here we combine it with best subsample selection to provide new robust methods. Let $d_{\mathrm K}$
denote the Kolmogorov distance between two c.d.f.'s, i.e., for $F,\ G\in\s{F}^p$, $$d_{\mathrm K}(F,G)=\sup_{x\in{\mathbb{R}}^p}|F(x)-G(x)|.$$ Take the objective
function in \eqref{bssp} as
\begin{equation}\label{robfn}\psi_n(\s{A})=\inf_{\theta\in\Theta}d_{\mathrm K}(\hat{H}_{\s{A}},F_\theta),\end{equation} where $\hat{H}_{\s{A}}$ is the empirical
distribution function based on the observations $\{X_i\}_{i\in\s{A}}$. We discuss BSP for this problem under model \eqref{cm2} through verifying the separation property
of $\{\psi_n\}$ in \eqref{robfn}.

\begin{assumption}\label{as:mine} For all $\alpha\in(0,\epsilon/(1-\tau)]$,
$\inf_{\theta\in\Theta}d_{\mathrm K}\big((1-\alpha)F_{\theta_0}+\alpha G,F_\theta\big)>0$.\end{assumption}

\begin{theorem} \label{th:bsk} Under Assumptions \ref{as:m} and \ref{as:mine}, $\{\psi_n\}$ in \eqref{robfn} separates $\mathfrak{A}$ from
$\mathfrak{B}$ in \eqref{robab}. \end{theorem}

\begin{proof} For $\{\s{A}_n\}\in\mathfrak{A}$, by the Glivenko-Cantelli theorem, \begin{equation}\label{fa1}\psi_n(\s{A}_n)
=\inf_{\theta\in\Theta}d_{\mathrm K}(\hat{H}_{\s{A}_n},F_\theta) \leqslant\inf_{\theta\in\Theta}d_{\mathrm K}(F_{\theta_0},F_\theta)+d_{\mathrm
K}(\hat{H}_{\s{A}_n},F_{\theta_0}) =d_{\mathrm K}(\hat{H}_{\s{A}_n},F_{\theta_0})\to0\ \text{(a.s.)}.
\end{equation}For $\{\s{A}_n\}\in\mathfrak{B}$, we have$$\psi_n(\s{A}_n)=\inf_{\theta\in\Theta}d_{\mathrm K}(\hat{H}_{\s{A}_n},F_\theta)
\geqslant\inf_{\theta\in\Theta}d_{\mathrm K}\big((1-\alpha)F_{\theta_0}+\alpha G,F_\theta\big)-d_{\mathrm K}\big(\hat{H}_{\s{A}_{n}},(1-\alpha)F_{\theta_0}+\alpha
G\big),$$ which implies\begin{equation}\label{fa2}\liminf_{n\to\infty}\psi_n(\s{A}_n) \geqslant\inf_{\theta\in\Theta}d_{\mathrm K}\big((1-\alpha)F_{\theta_0}+\alpha
G,F_\theta\big)>0\quad \text{(a.s.)}.\end{equation} Combining \eqref{fa1} and \eqref{fa2}, we complete the proof. \end{proof}

Compared to Assumption \ref{as:ss}, Assumption \ref{as:mine} is fairly weak. For example, let $F_\theta$ be the c.d.f. of $N(\theta,1)$ with $\theta_0=0$, and let
$\epsilon=\tau=1/2$. Suppose that $G$ is the c.d.f. of $N(\mu.\sigma^2)$. Assumption \ref{as:mine} holds for all $\sigma\neq1$.

As a byproduct, we next prove another interesting result that, with additional conditions, the estimator based on the best subsample selected by minimizing the objective
function $\psi_n$ in \eqref{robfn} is consistent even under the contaminated model \eqref{cm2}. This result provides further support of using this objective function. In
addition, to the best of the author's knowledge, this estimator is the first one that can converge to the true parameter even under the contaminated model, and may be of
independent interest.
\begin{assumption}\label{as:uf} For all $\theta_1,\theta_2\in\Theta$, $F_{\theta_1}=F_{\theta_2}$ implies $\theta_1=\theta_2$.\end{assumption}

\begin{assumption}\label{as:ft} For all $\phi\in\Theta$, $\lim_{\theta\to\phi}d_{\mathrm K}(F_\theta,F_{\phi})=0$.\end{assumption}

\begin{assumption}\label{as:fbp} For all $x\in\mathbb{R}^p$ and all $b\in\s{C}^*(\Theta)\setminus\Theta$, $\lim_{\theta\to b}\big(F_\theta(x)-F_\theta(-x)\big)=0$, where
$\s{C}^*(\Theta)$ is defined in Assumption \ref{as:g2}.
\end{assumption}

\begin{assumption}\label{as:id} For all $\theta\in\Theta,\ \theta\neq\theta_0$, $[(1-\epsilon)F_{\theta_0}+\epsilon G-\tau F_\theta]/(1-\tau)\notin\s{F}^p$.
\end{assumption}
Assumption \ref{as:id} is the key condition to guarantee that $\theta_0$ is estimable. Otherwise, if there exists $\theta_1\neq\theta_0$ such that
$U=[(1-\epsilon)F_{\theta_0}+\epsilon G-\tau F_{\theta_1}]/(1-\tau)\in\s{F}^p$, then \begin{equation*}\tau F_{\theta_1}+(1-\tau)U=(1-\epsilon)F_{\theta_0}+\epsilon
G,\end{equation*}which makes us unable to distinguish between $\theta_0$ and $\theta_1$. This assumption is stronger than Assumption \ref{as:mine}.

\begin{assumption}\label{as:gf} For all $\s{A}\in\mathfrak{D}_n$, $\arg\min_{\theta\in\Theta}
d_{\mathrm K}(\hat{H}_{\s{A}},F_\theta)$ exists (a.s.) for sufficiently large $n$.
\end{assumption}

\begin{lemma}\label{lemma:fn} Suppose that $F_n\in\s{F}^p$ for each $n$ with $d_{\mathrm K}(F_n,F)\to0$ as $n\to\infty$, where $F$ is a function defined on $\mathbb{R}^p$.
Then $F\in\s{F}^p$.
\end{lemma}

\begin{proof} We can prove this lemma by verifying the definition of a c.d.f. \end{proof}

Denote $\s{A}_n^*=\arg\min_{\s{A}\in\mathfrak{D}_n}\psi_n(\s{A})$. By Assumption \ref{as:gf}, $\hat{\theta}_{\s{A}_n^*}=\arg\min_{\theta\in\Theta} d_{\mathrm
K}(\hat{H}_{\s{A}_n^*},F_\theta)$ exists. We now present the consistency result of $\hat{\theta}_{\s{A}_n^*}$.

\begin{proposition}Suppose that Assumption \ref{as:m} and Assumptions \ref{as:uf}--\ref{as:gf} hold. Then \\(i) $\hat{\theta}_{\s{A}_n^*}\to\theta_0$
(a.s.); \\(ii) $d_{\mathrm K}(\hat{\theta}_{\s{A}_n^*},F_{\theta_0})\to0$ (a.s.).\end{proposition}

\begin{proof} Here we assume $\epsilon>0$. The proof for $\epsilon=0$ is similar and simpler. Denote $H=(1-\epsilon)F_{\theta_0}+\epsilon G$.

Let $\{\s{A}_n\}\in\mathfrak{A}$. We have
\begin{equation}\label{dc}d_{\mathrm K}(\hat{H}_{\s{A}_n^*},F_{\hat{\theta}_{\s{A}_n^*}}) \leqslant d_{\mathrm K}(\hat{H}_{\s{A}_n},F_{\hat{\theta}_{\s{A}_n}}) \leqslant
d_{\mathrm K}(\hat{H}_{\s{A}_n},F_{\theta_0})\to0\quad\text{(a.s.)}.\end{equation}On the other hand, letting $\mathbb{Z}_n$ denote $\{1,\ldots,n\}$, we have
\begin{eqnarray}&&d_{\mathrm K}\big(\tau\hat{H}_{\s{A}_n^*}+(1-\tau)\hat{H}_{\mathbb{Z}_n\setminus\s{A}_n^*},\ H\big)\nonumber
\\&\leqslant&d_{\mathrm K}\big(m\hat{H}_{\s{A}_n^*}/n+(n-m)\hat{H}_{\mathbb{Z}_n\setminus\s{A}_n^*}/n,\ H\big)\nonumber
\\&&\ +\,d_{\mathrm K}\big(\tau\hat{H}_{\s{A}_n^*}+(1-\tau)\hat{H}_{\mathbb{Z}_n\setminus\s{A}_n^*},\ m\hat{H}_{\s{A}_n^*}/n+(n-m)
\hat{H}_{\mathbb{Z}_n\setminus\s{A}_n^*}/n\big)
\nonumber\\&\leqslant& d_{\mathrm K}(\hat{H}_{\mathbb{Z}_n},\ H)+|\tau-m/n|+|(1-\tau)-(n-m)/n|\nonumber
\\&\to&0\quad\text{(a.s.)}.\label{tdc}\end{eqnarray}By \eqref{dc} and \eqref{tdc},
\begin{equation}\label{tdk}d_{\mathrm K}\big(\hat{H}_{\mathbb{Z}_n\setminus\s{A}_n^*},
\ (H-\tau F_{\hat{\theta}_{\s{A}_n^*}})/(1-\tau)\big)\to0\quad\text{(a.s.)}.\end{equation} Let $E\subset\Omega$ be the set where \eqref{dc} and \eqref{tdk} hold. We next
consider the convergence for a certain $\omega\in E$. Let $b$ be a limit point of $\hat{\theta}_{\s{A}_n^*}(\omega)$. Here we view $\infty$ as a limit point if
$\hat{\theta}_{\s{A}_n^*}(\omega)$ is unbounded.

Consider the case of $b\in\s{C}^*(\Theta)\setminus\Theta$. Let $\hat{\theta}_{\s{A}_{k_n}^*}(\omega)\to b$, where $\{k_n\}$ is a subsequence of $\{n\}_{n=1,2,\ldots}$.
For $\epsilon>0$, there exist $\delta>0$ and $x_0\in\mathbb{R}^p$ such that
\begin{equation}\label{hex}H(x_0)-H(-x_0)>1-\tau+\delta.\end{equation} By Assumption \ref{as:fbp}, for sufficiently large $n$,
$$F_{\hat{\theta}_{\s{A}_{k_n}^*}(\omega)}(x_0)-F_{\hat{\theta}_{\s{A}_{k_n}^*}(\omega)}(-x_0)<\delta/(3\tau).$$By \eqref{dc}, for sufficiently large $n$,
\begin{equation}\label{cto0}\hat{H}_{\s{A}_{k_n}^*(\omega)}(x_0)-\hat{H}_{\s{A}_{k_n}^*(\omega)}(-x_0)<\delta/(2\tau).\end{equation}
It follows from \eqref{tdk}, \eqref{hex}, and \eqref{cto0} that for sufficiently large $n$,
\begin{eqnarray*}&&\hat{H}_{\mathbb{Z}_n\setminus\s{A}_{k_n}^*(\omega)}(x_0)-\hat{H}_{\mathbb{Z}_n\setminus\s{A}_{k_n}^*(\omega)}(-x_0)
\\&>&\big[\big(H(x_0)-H(-x_0)\big)-\tau\big(\hat{H}_{\s{A}_{k_n}^*(\omega)}(x_0)-\hat{H}_{\s{A}_{k_n}^*(\omega)}(-x_0)\big)-\delta/2\big]/(1-\tau)\\&>&1.\end{eqnarray*}
This is a contradiction. Therefore, $b\in\Theta$.

By Assumption \ref{as:ft},
$$d_{\mathrm K}\big(\hat{H}_{\mathbb{Z}_{k_n}\setminus\s{A}_{k_n}^*(\omega)},\ [(1-\epsilon)F_{\theta_0}+\epsilon G-\tau F_b]/(1-\tau)\big)\to0$$
By Lemma \ref{lemma:fn}, \begin{equation}\label{ni}[(1-\epsilon)F_{\theta_0}+\epsilon G-\tau F_b]/(1-\tau)\in\s{F}^p.\end{equation}By Assumption \ref{as:id},
$b=\theta_0$. This completes the proof of (i), and (ii) follows from (i) immediately. \end{proof}

\begin{remark}From the above proof, when Assumption \ref{as:id} does not hold, any limit point of
$\hat{\theta}_{\s{A}_n^*}(\omega)$ satisfies \eqref{ni}. For small $\epsilon$, such $b$ cannot be far way from $\theta_0$ since
$d_{\mathrm{K}}(F_{b},F_{\theta_0})\leqslant\epsilon/(1-\epsilon)$.\end{remark}

\subsection{A simulation study}\label{subsec:robsimu}
\hskip\parindent \vspace{-0.6cm}

We conduct a small simulation study to compare the two subsample selection methods by likelihood and $d_{\mathrm{K}}$. Let the good observations $X_1,\ldots,X_{n_0}$ be
i.i.d. $\sim N(\theta,1)$ with $\theta_0=0$. We generate $n_\mathrm{o}$ outliers $X_{n_0+1},\ldots,X_{n_0+n_{\mathrm{o}}}$ as
\\(I): $X_{n_0+1}=\ldots=X_{n_0+n_{\mathrm{o}}}=1$;
\\(II): $X_{n_0+1},\ldots,X_{n_0+n_{\mathrm{o}}}$ i.i.d. $\sim N(1,0.5^2)$;
\\(III): $X_{n_0+1},\ldots,X_{n_0+n_{\mathrm{o}}}$ i.i.d. $\sim N(1,3)$.
\\We search the solutions to \eqref{bssp} with the objective functions \eqref{smle} and \eqref{robfn} through randomly generating $B$ subsets of size $m$, where $B$ is
varies from $10$ to $100$. As $B$ increases, the objective value becomes smaller and corresponds to a ``better" selector. In this simulation, we fix
$n=n_0+n_\mathrm{o}=20$, $m=10$, and consider two values of $n_\mathrm{o}$, 5 and 10. We Repeat 10,000 times to compute the mean objective values (MOVs) and MSEs. The
results are shown in Table \ref{tab:db}.

\begin{table}
\caption{\label{tab:db}Comparisons of different $B$'s in Section \ref{subsec:robsimu}} \centering\vspace{2mm}
{\begin{tabular}{*{9}{lclcccccc}}\hline&&\quad\quad&\quad\quad&\multicolumn{2}{c}{$n_o=5$}&\quad\quad&\multicolumn{2}{c}{$n_o=10$}
\\\cline{5-9}&&&&MOV&MSE&\quad\quad&MOV&MSE\\\hline
\multirow{4}{*}{(I)}
&&likelihood ($B=10$)      &&0.5078&0.2543&&0.3244&0.5712\\
&&likelihood ($B=100$)     &&0.3122&0.3904&&0.1347&0.7862\\
&&$d_{\mathrm K}$ ($B=10$) &&0.1437&0.1037&&0.2036&0.1694\\
&&$d_{\mathrm K}$ ($B=100$)&&0.1202&0.1012&&0.1677&0.1278\\\hline \multirow{4}{*}{(II)}
&&likelihood ($B=10$)      &&0.5484&0.2313&&0.4461&0.5021\\
&&likelihood ($B=100$)     &&0.3424&0.3052&&0.2542&0.6503\\
&&$d_{\mathrm K}$ ($B=10$) &&0.1267&0.1442&&0.1349&0.2660\\
&&$d_{\mathrm K}$ ($B=100$)&&0.1068&0.1438&&0.1134&0.2510\\\hline \multirow{4}{*}{(III)}
&&likelihood ($B=10$)      &&1.1135&0.1864&&2.1988&0.4447\\
&&likelihood ($B=100$)     &&0.5666&0.1738&&1.1227&0.3152\\
&&$d_{\mathrm K}$ ($B=10$) &&0.1319&0.2504&&0.1565&0.4366\\
&&$d_{\mathrm K}$ ($B=100$)&&0.1091&0.2429&&0.1232&0.3780\\\hline
\end{tabular}}
\end{table}

We state at the end of Section \ref{subsec:smle} that Assumption \ref{as:ss} does not hold when there are clustered outliers like (I) or (II), which can make BSP for
\eqref{smle} fail. The simulation results are consistent to this conclusion: the likelihood-based subsample estimator performs more poorly as $B$ increases. For this
case, the behavior of the $d_\mathrm{K}$-based subsample estimator follows BSP well: smaller MOV, smaller MSE. When the outliers are from (III), the two estimators both
follow BSP well, and the likelihood method is better.

\newpage
\section{Better subsample selection in regression}\label{sec:boer}
\hskip\parindent \vspace{-0.8cm}

This section discusses BSP for the best subsample problem in regression models. We show that the least trimmed squares (LTS) estimate (Rousseeuw 1984) is actually an
estimate based on the best subsample selected by the least squares, and prove the separation property of the corresponding objective function.

\subsection{Separation property of trimmed least squares}\label{subsec:robr}
\hskip\parindent \vspace{-0.6cm}

Consider a linear regression model
\begin{equation}\label{lm}\v{y}=\m{X}\vg{\beta}+\vg{\varepsilon},\end{equation}
where $\m{X}=(x_{ij})$ is the $n\times p$ regression matrix, $\v{y}=(y_1,\ldots,y_n)'\in{\mathbb{R}}^n$ is the response vector, $\vg{\beta}=(\beta_1,\ldots,\beta_p)'$ is
the vector of regression coefficients and $\vg{\varepsilon}=(\varepsilon_1,\ldots,\varepsilon_n)'$ is a vector of i.i.d. random errors with zero mean and finite variance
$\sigma^2$. The LTS estimate is a commonly used regression estimate with high breakdown value, and we describe it as follows. For any $\vg{\beta}$ in (\ref{lm}), denote
the corresponding residuals by $r_i(\vg{\beta})=y_i-\v{x}_i'\vg{\beta}$ for $i=1,\ldots,n$, where $\v{x}_i=(x_{i1},\ldots,x_{ip})'$. For a specified integer $m\leqslant
n$, the LTS estimator $\hat{\vg{\beta}}_{\mathrm{LTS}}$ is the solution to
\begin{equation}\min_{\scriptsize\vg{\beta}}\sum_{i=1}^mr_{\pi_i(\vg{\beta})}^2(\vg{\beta}),\label{lts}\end{equation}where
$r_{\pi_1(\vg{\beta})}^2(\vg{\beta})\leqslant\cdots\leqslant r_{\pi_n(\vg{\beta})}^2(\vg{\beta})$ are the ordered squared residuals. Denote
$J(\vg{\beta})=\{\pi_1(\vg{\beta}),\ldots,\pi_n(\vg{\beta})\}$.

For all $n$, let the decision space $\mathfrak{D}_n$ be the same as \eqref{d} in the previous section. Take the objective function as
\begin{equation}\label{robrfn}\psi_n(\s{A})=\|\v{y}_{\s{A}}-\m{X}_{[\s{A}]}\hat{\vg{\beta}}_{[\s{A}]}\|^2,\end{equation}where $\v{y}_{\s{A}}$ is the subvector of
$\v{y}$ corresponding to the subsample $\s{A}$, $\m{X}_{[\s{A}]}$ is the submatrix of $\m{X}$ corresponding to $\s{A}$, i.e., $\m{X}_{[\s{A}]}$ is obtained by removing
all the rows whose subscripts are not in $\s{A}$, and $\hat{\vg{\beta}}_{[\s{A}]}$ is the least squares estimator under $\s{A}$, i.e.,
$\hat{\vg{\beta}}_{[\s{A}]}=(\m{X}_{[\s{A}]}'\m{X}_{[\s{A}]})^{-1}\m{X}_{[\s{A}]}'\v{y}_{\s{A}}$.

We first show that the LTS estimator defined in \eqref{lts} corresponds to the solution that minimizes $\psi_n$ in \eqref{robrfn}. Denote
$$\s{A}_n^*=\arg\min_{\s{A}\in\mathfrak{D}_n}\psi_n(\s{A}).$$

\begin{proposition}The LTS estimator $\hat{\vg{\beta}}_{\mathrm{LTS}}$ satisfies $\hat{\vg{\beta}}_{\mathrm{LTS}}=\hat{\vg{\beta}}_{[\s{A}_n^*]}$ and
$J(\hat{\vg{\beta}}_{\mathrm{LTS}})=\s{A}_n^*$.\end{proposition}

\begin{proof} We have \begin{eqnarray*}&&\sum_{i=1}^mr_{\pi_i(\hat{\vg{\beta}}_{\mathrm{LTS}})}^2(\hat{\vg{\beta}}_{\mathrm{LTS}})
\leqslant\sum_{i=1}^mr_{\pi_i(\hat{\vg{\beta}}_{[\s{A}_n^*]})}^2(\hat{\vg{\beta}}_{[\s{A}_n^*]})
\leqslant\|\v{y}_{\s{A}_n^*}-\m{X}_{[\s{A}_n^*]}\hat{\vg{\beta}}_{[\s{A}_n^*]}\|^2
\\&&\leqslant\|\v{y}_{J(\hat{\vg{\beta}}_{\mathrm{LTS}})}-\m{X}_{[J(\hat{\vg{\beta}}_{\mathrm{LTS}})]}\hat{\vg{\beta}}_{[J(\hat{\vg{\beta}}_{\mathrm{LTS}})]}\|^2
\leqslant\|\v{y}_{J(\hat{\vg{\beta}}_{\mathrm{LTS}})}-\m{X}_{[J(\hat{\vg{\beta}}_{\mathrm{LTS}})]}\hat{\vg{\beta}}_{\mathrm{LTS}}\|^2
=\sum_{i=1}^mr_{\pi_i(\hat{\vg{\beta}}_{\mathrm{LTS}})}^2(\hat{\vg{\beta}}_{\mathrm{LTS}}),
\end{eqnarray*}which completes the proof.\end{proof}

We next discuss BSP for the optimization problem associated with LTS. Let $\s{A}_{0n}$ be the same as in \eqref{a0}. The contaminated regression model is assumed to be
\begin{equation}\label{cmlr}y_i=\vg{\beta}'\v{x}_i+\varepsilon_i\ \text{for}\ i\in\s{A}_{0n}\ \text{and}\ y_i=R(\v{x}_i)+\varepsilon_i\ \text{for}\
i\notin\s{A}_{0n},\end{equation} where $R$ is a function defined on $\mathbb{R}^p$, and $\varepsilon_1,\ldots,\varepsilon_n$ are the same as in model \eqref{lm}.

Some notation and assumptions are needed to prove the separation property of $\{\psi_n\}$ under model \eqref{cmlr}. Define $\mathfrak{S}_0$ and $\mathfrak{S}_1$ as in
\eqref{rob0}, and $\mathfrak{A}$ and $\mathfrak{B}$ as in \eqref{robab}. For all $\s{A}\in \mathfrak{D}_n$, let
$\m{H}_{[\s{A}]}=\m{I}_m-\m{X}_{[\s{A}]}(\m{X}_{[\s{A}]}'\m{X}_{[\s{A}]})^{-1}\m{X}_{[\s{A}]}'$ denote the projection matrix on the subspace $\{\v{x}\in\mathbb{R}^m:\
\m{X}_{[\s{A}]}'\v{x}=\v{0}\}$, where $\m{I}_m$ is the $m\times m$ identity matrix. In this section we let $\vg{\beta}$ itself denote the true parameter in model
\eqref{cmlr}, and assume that $p$ and $\vg{\beta}$ are fixed.

\begin{assumption} For all $\{k_n\}\in\mathfrak{S}_0\cup\mathfrak{S}_1$, $\m{X}_{\{k_n\}}'\m{X}_{\{k_n\}}/n\to$ a positive definite matrix as $n\to\infty$, where
$\m{X}_{\{k_n\}}=(\v{x}_{k_1}\ \cdots\ \v{x}_{k_n})'$. \label{as:sxx}
\end{assumption}

\begin{assumption} For all $\{k_n\}\in\mathfrak{S}_0\cup\mathfrak{S}_1$, $\v{v}'\m{H}_{\{k_n\}}\v{v}/n$ has a positive and finite limit as $n\to\infty$,
where $\v{v}=(v_1,\ldots,v_n)'$ with $v_i=0$ for $k_i\in\s{A}_{0n}$ and $v_i=R(\v{x}_{k_i})-\vg{\beta}'\v{x}_{k_i}$ otherwise,
$\m{H}_{\{k_n\}}=\m{I}_n-\m{X}_{\{k_n\}}(\m{X}_{\{k_n\}}'\m{X}_{\{k_n\}})^{-1}\m{X}_{\{k_n\}}'$, and $\m{X}_{\{k_n\}}$ is the same as in Assumption \ref{as:sxx}.
\label{as:sxx2}
\end{assumption}

\begin{remark}
If $\v{x}_1,\ldots,\v{x}_n$ are i.i.d. generated from a $p$-dimensional distribution with a positive definite covariance matrix, then Assumptions \ref{as:sxx} and
\ref{as:sxx2} hold (a.s.) providing $E(R(\v{x}_1)-\vg{\beta}'\v{x}_1)^2$ is positive and finite.\end{remark}

\begin{lemma}\label{lemma:lsc} Let $\{a_{nk}:\ k=1,\ldots,n,\ n=1,2,\ldots\}$ be an array of numbers satisfying $\sum_{k=1}^na_{nk}^2\leqslant1$. Then
$\sum_{k=1}^na_{nk}\varepsilon_i/\sqrt{n}\to0$ (a.s.) as $n\to\infty$.\end{lemma}

\begin{proof} See Chow (1966). \end{proof}

\begin{theorem} \label{th:rbsi} Under Assumptions \ref{as:sxx} and \ref{as:sxx2},
$\{\psi_n\}$ in \eqref{robrfn} separates $\mathfrak{A}$ from $\mathfrak{B}$ in \eqref{robab}. \end{theorem}

\begin{proof} For $\{\s{A}_n\}\in\mathfrak{A}$, we have \begin{equation}\label{rfa1}\psi_n(\s{A}_n)/m
=\|\v{y}_{\s{A}_n}-\m{X}_{[\s{A}_n]}\hat{\vg{\beta}}_{[\s{A}_n]}\|^2/m \leqslant\|\v{y}_{\s{A}_n}-\m{X}_{[\s{A}_n]}{\vg{\beta}}\|^2/m\to\sigma^2\
\text{(a.s.)}.\end{equation}For $\{\s{A}_n\}\in\mathfrak{B}$, denote $\s{A}_{1n}=\s{A}_n\cap\s{A}_{0n}$ and $\s{A}_{2n}=\s{A}_n\setminus\s{A}_{0n}$. Partition
$\m{H}_{[\s{A}_n]}$ as $\m{H}_{[\s{A}_n]}=(\m{H}_{[\s{A}_n]}^{(1)}\ \m{H}_{[\s{A}_n]}^{(2)})$, where $\m{H}_{[\s{A}_n]}^{(1)}$ corresponds to $\s{A}_{1n}$. We
have\begin{eqnarray*}&&\psi_n(\s{A}_n)=\|\v{y}_{\s{A}_n}-\m{X}_{[\s{A}_n]}\hat{\vg{\beta}}_{[\s{A}_n]}\|^2
=\left\|\m{H}_{[\s{A}_n]}\left(\begin{array}{l}\m{X}_{[\s{A}_{1n}]}\vg{\beta}\\R(\m{X}_{[\s{A}_{2n}]})\end{array}\right) +\m{H}_{[\s{A}_n]}\vg{\varepsilon}_{\s{A}_{n}}
\right\|^2
\\&=&\left\|\m{H}_{[\s{A}_n]}\left[\m{X}_{[\s{A}_{n}]}\vg{\beta}+
\left(\begin{array}{l}\m{0}\\R(\m{X}_{[\s{A}_{2n}]})-\m{X}_{[\s{A}_{2n}]}\vg{\beta}\end{array}\right)\right]
+\m{H}_{[\s{A}_n]}\vg{\varepsilon}_{\s{A}_{n}}\right\|^2
\\&=&\left\|\m{H}_{[\s{A}_n]}^{(2)}\big(R(\m{X}_{[\s{A}_{2n}]})-\m{X}_{[\s{A}_{2n}]}\vg{\beta}\big)+\m{H}_{[\s{A}_n]}\vg{\varepsilon}_{\s{A}_{n}}\right\|^2
\\&=&\vg{\varepsilon}_{\s{A}_{n}}'\m{H}_{[\s{A}_n]}\vg{\varepsilon}_{\s{A}_{n}}
+\big(R(\m{X}_{[\s{A}_{2n}]})-\m{X}_{[\s{A}_{2n}]}\vg{\beta}\big)'\m{H}_{[\s{A}_n]}^{(2)'}\m{H}_{[\s{A}_n]}\vg{\varepsilon}_{\s{A}_{n}}
\\&&\quad+\,\big(R(\m{X}_{[\s{A}_{2n}]})-\m{X}_{[\s{A}_{2n}]}\vg{\beta}\big)'\m{H}_{[\s{A}_n]}^{(2)'}\m{H}_{[\s{A}_n]}^{(2)}\big(R(\m{X}_{[\s{A}_{2n}]})
-\m{X}_{[\s{A}_{2n}]}\vg{\beta}\big),
\end{eqnarray*}where $R(\m{X}_{[\s{A}_{2n}]})=\big(R(\v{x}_i)\big)'_{i\in\s{A}_{2n}}$.
By Assumption \ref{as:sxx2}.
$\big(R(\m{X}_{[\s{A}_{2n}]})-\m{X}_{[\s{A}_{2n}]}\vg{\beta}\big)'\m{H}_{[\s{A}_n]}^{(2)'}\m{H}_{[\s{A}_n]}^{(2)}\big(R(\m{X}_{[\s{A}_{2n}]})
-\m{X}_{[\s{A}_{2n}]}\vg{\beta}\big)/m\to c>0$, which implies
$\big(R(\m{X}_{[\s{A}_{2n}]})-\m{X}_{[\s{A}_{2n}]}\vg{\beta}\big)'\m{H}_{[\s{A}_n]}^{(2)'}\m{H}_{[\s{A}_n]}\vg{\varepsilon}_{\s{A}_{n}}/m\to0$ (a.s.) by Lemma
\ref{lemma:lsc}. Note that $\vg{\varepsilon}_{\s{A}_{n}}'\m{H}_{[\s{A}_n]}\vg{\varepsilon}_{\s{A}_{n}}/m\to\sigma^2\ \text{(a.s.)}$ by Assumption \ref{as:sxx}. It
follows that\begin{equation}\label{rfa2}\psi_n(\s{A}_n)/m\to\sigma^2+c\ \text{(a.s.)}.
\end{equation}
Combining \eqref{rfa1} and \eqref{rfa2}, we complete the proof. \end{proof}

\subsection{A simulation study}\label{subsec:robrsimu}
\hskip\parindent \vspace{-0.6cm}

We conduct a small simulation to verify BSP for the LTS problem in finite-sample cases. Let $p=2$ and $\vg{\beta}=(0,0)'$ in model \eqref{lm}. Generate
$\{\v{x}_i=(x_{i1}, x_{i2})'\}_{i=1,\ldots,n}$ i.i.d. from a multivariate normal distribution $N(\v{0},\m{\Sigma})$ whose covariance matrix
$$\m{\Sigma}=\left(\begin{array}{cc}1&0.5\\0.5&1\end{array}\right).$$Then we obtain the regression matrix $\m{X}=(\v{x}_1,\ldots,\v{x}_{n})'$. There are $n_0$
observations that obey the linear relationship, i.e., $y_i=\vg{\beta}'\v{x}_i+\varepsilon_i$ for $i=1,\ldots,n_0$, where the random errors
$\varepsilon_1,\ldots,\varepsilon_{n_0}$ i.i.d. $\sim N(0,1)$. We generate $n_\mathrm{o}=n-n_0$ outliers as
\\(I): $y_i=5+\vg{\beta}'\v{x}_i+\varepsilon_i$ for $i=n_0+1,\ldots,n$, where the random errors $\varepsilon_{n_0+1},\ldots,\varepsilon_{n}$ i.i.d. $\sim N(0,1)$,
\\(II): $y_i=2x_{i1}-2x_{i2}+3x_{i1}^2+\varepsilon_i$ for $i=n_0+1,\ldots,n$, where the random errors $\varepsilon_{n_0+1},\ldots,\varepsilon_{n}$ i.i.d. $\sim N(0,3)$,
\\We search the solutions to minimize the objective functions \eqref{robrfn} through randomly generating $B$ subsets of size $m$. In this
simulation, we fix $n=20,\ n_0=15,\ m=11$ and consider three values of $B$, 100, 200, and 300. We repeat 10,000 times to compute the MOVs and MSEs as in Section
\ref{subsec:robsimu}. The results are shown in Table \ref{tab:dbr}. We can see that the results follow BSP well: as $B$ in increases, the MOV and MSE both decreases.

\begin{table}
\caption{\label{tab:dbr}Comparisons of different $B$'s in Section \ref{subsec:robrsimu}} \centering\vspace{2mm}
\begin{tabular}{*{8}{llcccccc}}\hline &&\quad\quad&\multicolumn{2}{c}{(I)}&\quad\quad&\multicolumn{2}{c}{(II)}
\\\cline{4-8}&&&MOV&MSE&\quad\quad&MOV&MSE\\\hline
\multirow{3}{*}{$B$}&$100$     &&12.464&1.2738&&6.9142&0.6723\\
                    &$200$     &&9.3486&0.9743&&5.7195&0.6589\\
                    &$300$     &&7.9660&0.8057&&5.2018&0.6479\\\hline
\end{tabular}
\end{table}

\section{Better subset for variable selection}\label{sec:bsr}
\hskip\parindent \vspace{-0.8cm}

Variable selection plays an important role in high-dimensional data analysis (B\"{u}hlmann and van de Geer 2011). Classical best subset regression ($\ell_0$-norm
regularized method) has been viewed as an infeasible method for moderate or large $p$, and other regularized methods with continuous penalties such as the nonnegative
garrote (Breiman, 1995), the lasso (Tibshirani, 1996), SCAD (Fan and Li, 2001), and MCP (Zhang, 2010) have become very popular in this area. However, Xiong (2014) showed
that, even for large $p$, best subset regression is still a valuable method since BSP for this problem, called the better-fitting better-screening rule in Xiong (2014),
holds under reasonable conditions. Therefore, we do not need to find the best subset (global solution), and a sub-optimal solution is usually satisfactory in practice.
In fact, Xiong (2014) proved the strong separation property of the objective function in best subset regression. In this section we continue discussing this problem for
both fixed and diverging $p$ cases.

\subsection{Selection for the fixed $p$ case}\label{subsec:fixp}
\hskip\parindent \vspace{-0.6cm}

Consider the linear regression model in \eqref{lm} with fixed $p$ and $\vg{\beta}$. Without loss of generality, assume that there is no intercept in \eqref{lm}, which
holds after standardizing $\m{X}$ and $\v{y}$. In this section, we denote the full model $\{1,\ldots,p\}$ by $\mathbb{Z}_p$. For $\s{A}\subset \mathbb{Z}_p$, $|\s{A}|$
denotes its cardinality, and $\m{X}_{\s{A}}$ denotes the submatrix of $\m{X}$ corresponding to $\s{A}$. As in Section \ref{sec:boer}, let $\vg{\beta}$ denote the true
parameter in model (\ref{lm}). Let $\s{A}_0$ denote the true submodel $\{j\in\mathbb{Z}_p:\ \beta_j\neq0\}$ with $d=|\s{A}_0|$. The decision space $\mathfrak{D}$ is the
power set of $\mathbb{Z}_p$, and its two subsets are
\begin{equation}\label{an0}\mathfrak{A}=\{\s{A}_0\},\ \mathfrak{B}=\mathfrak{D}\setminus\mathfrak{A}.\end{equation} We adopt the BIC
criterion (Schwarz, 1978) to select the important variables, which corresponds to the objective function
\begin{equation}\label{fn0}\psi_n(\s{A})=\big(1+|\s{A}|\log(n)/n\big)\|\v{y}-\m{X}_{\s{A}}\hat{\vg{\beta}}_{\s{A}}\big\|^2,\end{equation} where $\hat{\vg{\beta}}_{\s{A}}$
is the least squares estimator $(\m{X}_{\s{A}}'\m{X}_{\s{A}})^{-1}\m{X}_{\s{A}}'\v{y}$ under the submodel $\s{A}$. It is known that minimizing the BIC criterion leads to
consistent variable selection for the fixed $p$ case; see e.g., Shao (1997). Here we provide a stronger result that BSP for this optimization problem holds through
proving the strong separation property of $\psi_n$ in \eqref{fn0}. Our result indicates that, for two subsets, the better one having smaller BIC value is more likely to
be the true submodel.

Some notation and an assumption are needed. For $\s{A}\in \mathfrak{D}$, let $\m{H}_{\s{A}}=\m{I}_n-\m{X}_{\s{A}}(\m{X}_{\s{A}}'\m{X}_{\s{A}})^{-1}\m{X}_{\s{A}}'$ denote
the projection matrix on the subspace $\{\v{x}\in\mathbb{R}^n:\ \m{X}_{\s{A}}'\v{x}=\v{0}\}$. We denote by $\lambda_{\min}(\cdot)$ the smallest eigenvalue of a matrix.
Let $\beta_{\min}$ denote the component of $\vg{\beta}_{\s{A}_0}$ that has the smallest absolute value. Define
$$\delta_n=\min_{\s{A}_0\setminus\s{A}\neq\emptyset}\left[\frac{1}{n}\lambda_{\min}(\m{X}_{\s{A}_0\setminus\s{A}}' \m{H}_{\s{A}}\m{X}_{\s{A}_0\setminus\s{A}})\right].$$

\begin{assumption} As $n\to\infty$, $\m{X}'\m{X}/n\to\m{\Sigma}$, where $\m{\Sigma}$ is a positive definite matrix. \label{as:xx}\end{assumption}

This assumption is a standard condition to handel fixed $p$ asymptotics in linear regression (Gleser 1965; Knight and Fu 2000).

\begin{theorem} \label{th:bsr} Under Assumption \ref{as:xx}, $\{\psi_n\}$ in \eqref{fn0} strongly separates $\mathfrak{A}$ from $\mathfrak{B}$ in \eqref{an0}.
\end{theorem}

\begin{proof} First we have $\psi_n(\s{A}_0)/n\to\sigma^2$ (a.s.), For $\s{A}\in\mathfrak{D}$ with $\s{A}_0\setminus\s{A}\neq\emptyset$,
\begin{eqnarray*}&&\big(1+|\s{A}|\log(n)/n\big)^{-1}\psi_n(\s{A})\\&=&\vg{\varepsilon}'\m{H}_{\s{T}_n}\vg{\varepsilon}
+2\vg{\beta}_{\s{A}_0}'\m{X}_{\s{A}_0}'\m{H}_{\s{A}}\vg{\varepsilon} +\vg{\beta}_{\s{A}_0}'\m{X}_{\s{A}_0}'\m{H}_{\s{A}}\m{X}_{\s{A}_0}\vg{\beta}_{\s{A}_0}
\\&=&\vg{\varepsilon}'\m{H}_{\s{T}_n}\vg{\varepsilon}
+2\vg{\beta}_{\s{A}_0}'\m{X}_{\s{A}_0}'\m{H}_{\s{A}}\vg{\varepsilon}+\vg{\beta}_{\s{A}_0\setminus\s{A}}'\m{X}_{\s{A}_0\setminus\s{A}}'
\m{H}_{\s{A}}\m{X}_{\s{A}_0\setminus\s{T}_n}\vg{\beta}_{\s{A}_0\setminus\s{A}}
\\&\geqslant&\vg{\varepsilon}'\m{H}_{\s{A}}\vg{\varepsilon}
+2\vg{\beta}_{\s{A}_0}'\m{X}_{\s{A}_0}'\m{H}_{\s{A}}\vg{\varepsilon}+n\delta_n|\beta_{\min}|^2.\end{eqnarray*}Note that
$E(\vg{\beta}_{\s{A}_0}'\m{X}_{\s{A}_0}'\m{H}_{\s{A}}\vg{\varepsilon})^2/n^2
=\sigma^2\mathrm{tr}(\vg{\beta}_{\s{A}_0}'\m{X}_{\s{A}_0}'\m{H}_{\s{A}}\m{X}_{\s{A}_0}\vg{\beta}_{\s{A}_0})/n^2\to0$,
$\vg{\varepsilon}'\m{H}_{\s{A}}\vg{\varepsilon}/n\to\sigma^2$ (a.s.), and $\delta_n$ has a positive limit point. It follows
that\begin{equation}\label{a1}P\big(\psi_n(\s{A}_0)<\psi_n(\s{A})\big)\to1.\end{equation}

For $\s{A}\supset\s{A}_0$ with $\s{A}\neq\s{A}_0$,\begin{eqnarray*}&&\psi_n(\s{A}_0)-\psi_n(\s{A})=\vg{\varepsilon}'(\m{H}_{\s{A}_0}-\m{H}_{\s{A}})\vg{\varepsilon}
+\vg{\varepsilon}'\m{H}_{\s{A}_0}\vg{\varepsilon}|\s{A}_0|\log(n)/n-\vg{\varepsilon}'\m{H}_{\s{A}}\vg{\varepsilon}|\s{A}|\log(n)/n.\end{eqnarray*}Note that
$\m{H}_{\s{A}_0}-\m{H}_{\s{A}}$ converges to an idempotent matrix of rank $|\s{A}|-|\s{A}_0|$, which implies
$\vg{\varepsilon}'(\m{H}_{\s{A}_0}-\m{H}_{\s{A}})\vg{\varepsilon}=O_p(1)$.
Therefore,\begin{equation}\label{a2}\big(\psi_n(\s{A}_0)-\psi_n(\s{A})\big)/\log(n)\to-\sigma^2(|\s{A}|-|\s{A}_0|)\quad\text{in probability}.\end{equation}  Note that
$\mathfrak{D}$ is a finite set. By \eqref{a1} and \eqref{a2}, \begin{eqnarray*}&&P\left(\psi_n(\s{A}_0)<\inf_{\s{A}\in\mathfrak{B}}\psi_n(\s{A})\right)\\&\geqslant&
P\left(\psi_n(\s{A}_0)<\inf_{\s{A}_0\setminus\s{A}\neq\emptyset}\psi_n(\s{A})\right)+P\left(\psi_n(\s{A}_0)<\inf_{\s{A}\supset\s{A}_0,\
\s{A}\neq\s{A}_0}\psi_n(\s{A})\right)-1\\&\to&1,\end{eqnarray*} which completes the proof.\end{proof}

\begin{remark}With almost the same proof, Theorem \ref{th:bsr} also holds for the objective function
\begin{equation*} \psi_n(\s{A})=\big(1+|\s{A}|\lambda_n/n\big)\|\v{y}-\m{X}_{\s{A}}\hat{\vg{\beta}}_{\s{A}}\big\|^2 \end{equation*}
with $\lambda_n\to\infty$ and $\lambda_n/n\to0$, which corresponds to the GIC criterion (Rao and Wu 1989). BIC is its special case corresponding to
$\lambda_n=\log(n)$.\end{remark}

\subsection{Screening for the diverging $p$ case}\label{subsec:bsr1}
\hskip\parindent \vspace{-0.6cm}

When $p$ increases faster than $n$, it becomes more difficult to find consistent variable selection procedures. A compromised strategy is to use a two-stage procedure
(Fan and Lv 2008). In the first stage, a screening approach is applied to pick $M$ variables, where $M<n$ is specified. In the second stage, the coefficients in the
screened $M$-dimensional submodel can be estimated by well-developed regression techniques. To guarantee the effectiveness of this procedure, the first stage should
possess the sure screening property, i.e., it should retain all important variables in the model asymptotically (Fan and Lv 2008). A number of screening procedures have
been studied in the literature; see Fan and Lv (2008), Hall and Miller (2009), Fan et al. (2009), Wang (2009), and Li et al. (2012), among others. Following Xiong
(2014), in this subsection we establish BSP for best subset regression in the screening problem.

The model and related notation are the same as in Section \ref{subsec:fixp} except that $p$ and the true parameter $\vg{\beta}$ can depend on $n$, We let $\s{A}_{0n}$
denote the true submodel, which depends on $n$ as well. For a specified $M$ with $d\leqslant M<n$, the decision space is $\mathfrak{D}_n=\{\s{A}\subset\mathbb{Z}_p:\
|\s{A}|=M\}$, and two decision subsets are
\begin{equation*}\mathfrak{A}_n=\{\s{A}\subset\mathfrak{D}_n:\ \s{A}\supset\s{A}_0\},\ \mathfrak{B}_n=\mathfrak{D}_n\setminus\mathfrak{A}_n.\end{equation*}Denote
\begin{equation}\label{an}\mathfrak{A}=\prod_{n=1}^\infty\mathfrak{A}_n,\ \mathfrak{B}=\prod_{n=1}^\infty\mathfrak{B}_n.\end{equation}
The objective function is
\begin{equation}\label{fn}\psi_n(\s{A})=\|\v{y}-\m{X}_{\s{A}}\hat{\vg{\beta}}_{\s{A}}\big\|^2,\end{equation} where $\hat{\vg{\beta}}_{\s{A}}$ is the least squares estimator
$(\m{X}_{\s{A}}'\m{X}_{\s{A}})^{-}\m{X}_{\s{A}}'\v{y}$ under the submodel $\s{A}$. Here we allow $\m{X}_{\s{A}}$ not to be of full column rank, and therefore the
generalized inverse ``$^{-}$" is used in $\hat{\vg{\beta}}_{\s{A}}$. Similar to Section \ref{subsec:fixp}, for $\s{A}\in \mathfrak{D}_n$, denote
$\m{H}_{\s{A}}=\m{I}_n-\m{X}_{\s{A}}(\m{X}_{\s{A}}'\m{X}_{\s{A}})^{-}\m{X}_{\s{A}}'$ and
$$\delta_n=\min_{\s{A}\notin \mathfrak{A}_n}\left[\frac{1}{n}\lambda_{\min}(\m{X}_{\s{A}_{0n}\setminus\s{A}}' \m{H}_{\s{A}}\m{X}_{\s{A}_{0n}\setminus\s{A}})\right].$$

Note that the objective function $\psi_n$ in \eqref{fn} is the residual sum of squares, which describes the fit of a submodel $\s{A}$. Based on this, Xiong (2014)
provided the better-fitting better-screening rule for screening important variables, i.e., a better subset with better fit is more likely to include all important
variables. This rule is actually the BSP for the problem of minimizing $\psi_n$ in \eqref{fn}, and follows from the following strong separation property of $\psi_n$
proved by Xiong (2014).

\begin{assumption}The random error $\vg{\varepsilon}$ in (\ref{lm}) follows a normal distribution $N(\v{0}, \sigma^2\m{I}_n)$. \label{as:norm}\end{assumption}

\begin{assumption}There exists a constant $C>0$ such that $\sum_{i=1}^nx_{ij}^2/n\leqslant C$ for any $j\in\s{A}_0$.\label{as:stand}\end{assumption}

\begin{assumption}\label{as:bsr} As $n\rightarrow\infty$, $(\delta_n\beta_{\min}^2)^{-1}=O(n^{\gamma_1}),\
\|\vg{\beta}\|(\delta_n\beta_{\min}^2)^{-1}=O(n^{\gamma_2}),\ d=O(n^{\gamma_3}),\ M=O(n^{\gamma_4})$, and $\log p=O(n^{\gamma_5})$, where $\gamma_i\geqslant 0\
(i=1,\ldots,5)$, $2\gamma_1+\gamma_4+\gamma_5<1$, and $2\gamma_2+\gamma_3+\gamma_4+\gamma_5<1$.\label{as:me}\end{assumption}

\begin{theorem} \label{th:bsr} Under Assumptions \ref{as:norm}--\ref{as:bsr}, $\{\psi_n\}$ in \eqref{fn} strongly
separates $\mathfrak{A}$ from $\mathfrak{B}$ in \eqref{an}.
\end{theorem}

Assumption \ref{as:bsr} is strong in that $M$ cannot be too large, whereas in practice we usually use a large $M=O(n^\gamma)$ with an unrestrictive $\gamma\in(0,1)$, or
even $M=O(n)$, for insurance. We next show that, under fairly weak conditions, $\{\psi_n\}$ has the weak separation property. By Theorem \ref{th:lt2}, Corollary
\ref{cor:w}, and Remark \ref{rk:ctb2}, the weak separation property suffices to imply the better-fitting better-screening rule for practice use.

\begin{assumption}\label{as:bsrw} As $n\to\infty$, $M=o(n\delta_n\beta_{\min}^2)$, and for all $\epsilon>0$,
$\sum_{n=1}^\infty\exp(-\epsilon n\delta_n\beta_{\min}^2)<\infty$, $\sum_{n=1}^\infty\exp\big(-\epsilon
n\delta_n^2\beta_{\min}^4/(d\|\vg{\beta}\|^2)\big)<\infty$.\end{assumption}

\begin{assumption}\label{as:bsrw2} As $n\to\infty$, $M/n\to\alpha\in(0,1)$ and ${\mathrm{rank}}(\m{H}_{\s{A}})/n\to\lambda\in[0,1)$ for all
$\{\s{A}_n\}\in \mathfrak{A}\cup\mathfrak{B}$; for all $n$, $\delta_n\beta_{\min}^2\geqslant c$, where $c>0$ is a constant; for all $\epsilon>0$,
$\sum_{n=1}^\infty\exp\big(-\epsilon n/(d\|\vg{\beta}\|^2)\big)<\infty$. \end{assumption}

\begin{lemma}\label{lemma:chi}(i) Let $\xi_n\sim N(0,1)$ for all $n$. Suppose that $b_n$ satisfies $\sum_{n=1}^\infty\exp(-\epsilon b_n^2)<\infty$ for all $\epsilon>0$.
Then $\xi_n/b_n\to0$ (a.s.).
\\(ii) Let $\xi_n\sim\chi_{r_n}^2$, where $r_n$ is a positive
integer for all $n$. Suppose that $b_n$ satisfies $r_n/b_n\to\alpha\in[0,1]$ and $\sum_{n=1}^\infty\exp(-\epsilon b_n)<\infty$ for all $\epsilon>0$. Then
$(\xi_n-r_n)/b_n\to0$ (a.s.).\end{lemma}

\begin{proof} (i) By the Borel-Cantelli lemma, it suffices to show that, for all $\epsilon>0$,
\begin{equation}\label{cc}\sum_{n=1}^\infty P(|\xi_n/b_n|>\epsilon)<\infty.\end{equation}Let $\Phi$ denote the c.d.f. of $\xi_n$. We have
$P(|\xi_n/b_n|>\epsilon)=2[1-\Phi(\epsilon b_n)]\leqslant(\epsilon b_n)^{-1}\exp(-\epsilon^2 b_n^2/2)\leqslant\exp(-\epsilon^2 b_n^2/2)$ for sufficiently large $n$,
which implies \eqref{cc}.
\\(ii) By Lemma 1 in Xiong (2014), for all $\epsilon>0$,
$$P(|\xi_n-r_n|/b_n>\epsilon)=P(|\xi_n/r_n-1|>\epsilon b_n/r_n)\leqslant2\exp\left(-\frac{\epsilon^2}{4}b_n(1+r_n/b_n)^{-1}\right).$$
By the Borel-Cantelli lemma, we complete the proof.\end{proof}

\begin{theorem} \label{th:bsr} Under Assumption  \ref{as:norm} and \ref{as:stand}, if Assumption \ref{as:bsrw} or \ref{as:bsrw2} holds, then $\{\psi_n\}$ in \eqref{fn}
separates $\mathfrak{A}$ from $\mathfrak{B}$ in \eqref{an}. \end{theorem}

\begin{proof} If Assumption \ref{as:bsrw} holds, It suffices to show that, for any $\{\s{A}_n\}\in \mathfrak{A}$
and $\{\s{B}_n\}\in \mathfrak{B}$,
\begin{equation}\label{fs}\limsup_{n\to\infty}\left[\psi_n(\s{A}_n)/(n\delta_n\beta_{\min}^2)-\psi_n(\s{B}_n)/(n\delta_n\beta_{\min}^2)\right]<0\quad \text{(a.s.)}.
\end{equation}We have
\begin{eqnarray*}&&\psi_n(\s{A}_n)-\psi_n(\s{B}_n)\\&=&\vg{\varepsilon}'\m{H}_{\s{A}_n}\vg{\varepsilon}-(\vg{\varepsilon}'\m{H}_{\s{B}_n}\vg{\varepsilon}
+2\vg{\beta}_{\s{A}_{0n}}'\m{X}_{\s{A}_{0n}}'\m{H}_{\s{B}_n}\vg{\varepsilon}
+\vg{\beta}_{\s{A}_{0n}}'\m{X}_{\s{A}_{0n}}'\m{H}_{\s{B}_n}\m{X}_{\s{A}_{0n}}\vg{\beta}_{\s{A}_{0n}})
\\&=&\vg{\varepsilon}'(\m{I}_n-\m{H}_{\s{A}_n})\vg{\varepsilon}-\vg{\varepsilon}'(\m{I}_n-\m{H}_{\s{B}_n})\vg{\varepsilon}
-2\vg{\beta}_{\s{A}_{0n}}'\m{X}_{\s{A}_{0n}}'\m{H}_{\s{B}_n}\vg{\varepsilon}-\vg{\beta}_{\s{A}_{0n}\setminus\s{B}_n}'\m{X}_{\s{A}_{0n}\setminus\s{B}_n}'
\m{H}_{\s{B}_n}\m{X}_{\s{A}_{0n}\setminus\s{B}_n}\vg{\beta}_{\s{A}_{0n}\setminus\s{B}_n}
\\&\leqslant&\vg{\varepsilon}'(\m{I}_n-\m{H}_{\s{A}_n})\vg{\varepsilon}-\vg{\varepsilon}'(\m{I}_n-\m{H}_{\s{B}_n})\vg{\varepsilon}
-2\vg{\beta}_{\s{A}_{0n}}'\m{X}_{\s{A}_{0n}}'\m{H}_{\s{B}_n}\vg{\varepsilon}-n\delta_n\beta_{\min}^2.\end{eqnarray*}By Lemma \ref{lemma:chi} (ii) and Assumption
\ref{as:bsrw}, $[\vg{\varepsilon}'(\m{I}_n-\m{H}_{\s{A}_n})\vg{\varepsilon}-\vg{\varepsilon}'(\m{I}_n-\m{H}_{\s{B}_n})\vg{\varepsilon}]/(n\delta_n\beta_{\min}^2)\to0$
(a.s.) Note that $\vg{\beta}_{\s{A}_{0n}}'\m{X}_{\s{A}_{0n}}'\m{H}_{\s{A}}\vg{\varepsilon}\sim N(0,v^2)$, where $v^2=\sigma^2\vg{\beta}_{\s{A}_{0n}}'\m{X}_{\s{A}_{0n}}'
\m{H}_{\s{A}}\m{X}_{\s{A}_{0n}}\vg{\beta}_{\s{A}_{0n}}$. By Assumption \ref{as:stand}, $v^2\leqslant
\sigma^2\lambda_{\max}(\m{H}_{\s{A}})\lambda_{\max}(\m{X}_{\s{A}_{0n}}'\m{X}_{\s{A}_{0n}})\|\vg{\beta}\|^2 \leqslant
\sigma^2{\mathrm{tr}}(\m{X}_{\s{A}_{0n}}'\m{X}_{\s{A}_{0n}})\|\vg{\beta}\|^2\leqslant nCd\sigma^2\|\vg{\beta}\|^2$. By Lemma \ref{lemma:chi} (i) and Assumption
\ref{as:bsrw}, $\vg{\beta}_{\s{A}_{0n}}'\m{X}_{\s{A}_{0n}}'\m{H}_{\s{A}}\vg{\varepsilon}/(n\delta_n\beta_{\min}^2)\to0$ (a.s.). This completes the proof of \eqref{fs}.

If Assumption \ref{as:bsrw2} holds, similar to the above proof, we can show$$\limsup_{n\to\infty}\left[\psi_n(\s{A}_n)/n-\psi_n(\s{B}_n)/n\right]<0\quad
\text{(a.s.)},$$which completes this proof.\end{proof}

\begin{remark}It is worthwhile noting that there is no any restriction on $p$ in Theorem \ref{th:bsr}. That is to say, if the required conditions are satisfied, then
Theorem \ref{th:bsr} holds no matter how large $p$ is. This point seems interesting since almost all results on high-dimensional asymptotics in the literature require
$p=o(\exp(n))$ (B\"{u}hlmann and van de Geer 2011). \end{remark}

\subsection{A simulation study}\label{subsec:bsrsimu}
\hskip\parindent \vspace{-0.6cm}

We conduct a small simulation study to verify the better-fitting better-screening rule. Related simulation results can be found in Xiong (2014). In model \eqref{lm}, all
rows of $\m{X}$ are i.i.d. from a multivariate normal distribution $N(\v{0},\m{\Sigma})$ whose covariance matrix $\m{\Sigma}=(\sigma_{ij})_{p\times p}$ has entries
$\sigma_{ii}=1,\ i=1,\ldots,p$ and $\sigma_{ij}=\rho,\ i\neq j$. The coefficients are given by $\beta_1=\beta_2=\beta_3=3$ and $\beta_j=0$ for other $j$. The random
errors $\varepsilon_1,\ldots,\varepsilon_n$ i.i.d. $\sim N(0,1)$, We fix $n=50$ and $\rho=0.05$, and vary $p$ from 100 to 10000. Three screening methods with $M=25$ are
compared: Efron et al. (2004)'s least angle regression (LAR), Fan and Lv (2008)'s sure independence screening (SIS), and the ``better" method that uses the better
results produced by LAR and SIS with smaller residual sum of squares as the final submodel. For each model, we simulate 1000 data sets and compute the coverage rates
(CRs) of including the true submodel, which are displayed in Table \ref{tab:screening}. We can see that all the results follow the better-fitting better-screening rule
well: the ``better" solution always yields larger CRs than LAR and SIS.

\begin{table}
\caption{\label{tab:screening}CR comparisons in Section \ref{subsec:bsrsimu}} \centering\vspace{2mm}
\begin{tabular}{*{8}{lccccccc}}\hline&\quad\quad&\multicolumn{6}{c}{$p$}\\\cline{3-8}&&$100$&$500$&$1000$&$3000$&$5000$&$10000$\\\hline
LAR    &&0.999&0.931&0.845&0.656&0.552&0.434\\
SIS    &&0.999&0.977&0.955&0.892&0.820&0.728\\
better &&1    &0.989&0.961&0.906&0.832&0.737\\\hline
\end{tabular}
\end{table}

\section{Discussion}\label{sec:dis}
\hskip\parindent
\vspace{-0.8cm}

In this section we end this paper with some discussion.

\subsection{Summary}
\hskip\parindent \vspace{-0.6cm}

When the global solution to a statistical optimization problem is difficult to obtain, BSP theoretically supports to the method of using the solution whose objective
value is as small as possible (for minimization problems). Interestingly, it can be studied within a simple framework based on several obvious but effective comparison
theorems. These theorems tell us that a better solution with smaller objective value is more likely to be a good decision if the objective function has the (strong)
separation property. Therefore, it suffices to prove the separation property of the objective function for verifying BSP. Following this way, we have discussed BSP for
several statistical optimization problems, and have established the corresponding separation properties. These problems lead to basic but important statistical methods,
including maximum likelihood estimation, best subsample selection in robust statistics, and best subset regression in variable selection.

Besides the usefulness in theory, BSP can provide viewpoints on the development of methodologies. In Section \ref{sec:boe}, a new best subsample selection method based
on the Kolmogorov distance has been introduced. The corresponding BSP holds under fairly weak conditions. Theoretical and numerical studies both show that this method
perform well when there are clustered outliers. As a byproduct, the robust estimate based on this selection method is consistent even under contaminated models. This
estimate may be of independent interest in robust estimation.

Strictly speaking, the strong separation property of the objective function is needed to establish BSP. This is a strong condition and actually implies the consistency
of the global solution (Theorems \ref{th:con} and \ref{th:con2}). We have proved this property only for maximum likelihood problem and best subset regression under
strong conditions. If we ignore the mathematical details, the weak separation property seems enough for practical use (Remarks \ref{rk:ctb} and \ref{rk:ctb2}). In
general, the weak separation property is relatively easy to prove. Simulation results in this paper are consistent to the theoretical discoveries even when only the weak
separation property is proved.

When computing the global solution is a problem, we should consider whether BSP holds. This principle is as important as the consistency property or other properties of
the global solution. We hope that statisticians will always keep BSP in mind when handling complex optimization problems. On the other hand, BSP can be used to justify a
statistical method from an optimization problem. A good objective function whose separation properties hold under mild conditions can provide us a way to combine weak
methods into a stronger one. Such examples can be found in Sections \ref{subsec:smlesimu} and \ref{subsec:bsrsimu}: the ``better" method can improve weak methods through
comparing their objective values.

\subsection{Limitations of this paper}
\hskip\parindent \vspace{-0.6cm}

A prerequisite of BSP is that the global solution has, or, is at least expected to have, desirable statistical properties. Therefore, the BSP theory is not applicable to
the statistical methods which are ``irregularly" derived from optimization problems. An example is boosting. Some authors showed that boosting can be viewed as a
steepest descent algorithm for minimizing a loss function (Breiman 1998; Friedman, Hastie, and Tibshirani 2000). It is stopped early since the minimum usually leads to
overfitting. In other words, minimization here is a ``pretense", and we are really interested in the solutions along the path to the minimum, not the minimum itself.
Another example is SCAD (Fan and Li 2001), which is a penalized likelihood estimate with a nonconcave penalty. Fan and Li (2001) proved that there exists a local
solution of SCAD possessing the so-called oracle property. When the oracle property is concerned, the local solution with this property is preferred to the global
solution, and thus BSP fails.

BSP is a general and non-specific concept, since the statistical properties of the solution to a optimization problem can be multifold. For example, besides estimation
accuracy, we use the M-estimate because of its robust properties such as the minimax property (Huber 1981). Another example is the regularized least squares method for
regression models such as the lasso (Tibshirani 1996), which is used for simultaneous estimation and variable selection. Therefore, we should evaluate it in terms of
both estimation and selection performance. This paper is just a beginning of the study on BSP, and focuses on the statistical property described with the probability of
being a ``good" decision. Nevertheless, we believe that there are other reasonable frameworks to establish BSP, which describe the ``better" statistical properties,
maybe non-asymptotics, in different manners and/or can cover multifold statistical properties of interest.

This paper does not discuss algorithms, i.e., how to find a better solution when BSP holds. For the algorithms used in best subsample and best subset selection, we refer
the reader to Rousseeuw and Van Driessen (1999), Hawkins and Olive (2002), Rousseeuw and Van Driessen (2006), Hofmann, Gatu, and Kontoghiorghes (2007), and Xiong (2014).
For discussion on general global optimization algorithms in statistics, see, e.g., Fang, Hickernell, and Winker (1996).

\subsection{Future directions}
\hskip\parindent \vspace{-0.6cm}

Besides the limitations of this paper aforementioned, a number of issues on BSP and related topics seem valuable to research in the future.

The applications of the comparison theorems presented in Sections \ref{sec:smle}-\ref{sec:bsr} are selective. The range of potential applications can be much broader
than presented. For example, many optimization problems listed in Section \ref{sec:intro} can be studied using them. A number of useful theoretical results that can
guide real data analysis may be obtained along this direction.

Another research direction related to BSP is to study statistical properties of sub-optimal solutions produced by certain algorithms. Recently, Ma, Mahoney, and Yu
(2013) studied estimation accuracy of several leverage-based algorithms for large-scale least squares problems. For the SCAD problem aforementioned, the statistical
properties of its local solutions that can be achieved by certain algorithms were discussed by Loh and Wainwright (2013), Wang, Liu, and Zhang (2013), and Xiong, Dai,
and Qian (2013). Note that forward stepwise selection is a greedy algorithm for best subset regression (Miller 2002). The study on its screening properties (Wang 2009)
can be bracketed with the work of this kind. For some estimation problems, the estimators derived from only one iteration of certain iterative algorithms can also have
appealing properties with good starting estimators (Bickel 1975; Fan and Chen 1999; Zou and Li 2008). Perhaps it is also valuable to study algorithm-based BSP.

Recently, Big Data begins to pose significant challenges to statistics (Fan, Han, and Liu 2013). For analyzing Big Data, not only statistical methodology but also
statistical theory should be considered based on computation. BSP can be viewed as a computational ability-based statistical theory, and we expect that BSP and related
methodologies will be paid more attention to in the future.

\section*{Acknowledgements}
\hskip\parindent \vspace{-0.8cm}

This work is supported by the National Natural Science Foundation of China (Grant No. 11271355). The author thanks Professor C. F. Jeff Wu for helpful discussion. The
author is also grateful to the support of Key Laboratory of Systems and Control, Chinese Academy of Sciences.

\vspace{1cm}
\noindent{\Large\bf References}

{\begin{description}

\item{}
Agull\'{o}, J., Croux, C., and Van Aelst, S. (2008). The multivariate least-trimmed squares estimator. \textit{Journal of Multivariate Analysis} {\bf 99} 311--338.

\item{}
Atkinson, A. C., Donev, A. N., and Tobias, R. D. (2007). \textit{Optimum Experimental Designs, with SAS}. Oxford University Press, Oxford.


\item{}
Anderson, T. W. (2003) \textit{An Introduction to Multivariate Statistical Analysis, 3rd Edition}. Wiley, New York.



\item{}
Bickel, P. J. (1975). One-step Huber estimates in the linear model. \textit{Journal of the American Statistical Association} {\bf 70} 428--434.


\item{}
Breiman, L. (1995). Better subset regression using the nonnegative garrote. \textit{Technometrics} {\bf 37} 373--384.

\item{}
Breiman, L. (1998). Arcing classifiers (with discussion). \textit{The Annals of Statistics} {\bf 26} 801--849.

\item{}
B\"{u}hlmann, P. and Hothorn, T. (2007). Boosting algorithms: regularization, prediction and model fitting. \textit{Statistical Science} {\bf 22} 477--505.

\item{}
B\"{u}hlmann, P. and van de Geer, S. (2011). \textit{Statistics for High-Dimensional Data: Methods, Theory and Applications}. Springer, New York.

\item{}
Butler, R. W., Davies, P. L., and Jhun, M. (1993). Asymptotics for the minimum covariance determinant estimator. \textit{The Annals of Statistics} \textbf{21}
1385--1400.


\item{}
Chow, Y. S. (1966). Some convergence theorems for independent random variables. \textit{The Annals of Mathematical Statistics} {\bf 37} 1482--1493.


\item{}
Donoho, D. L. and Liu, R. C. (1994). The ``automatic" robustness of minimum distance functional. \textit{The Annals of Statistics} \textbf{16} 552--586.

\item{}
Dorsey, R. E. and Mayer, M. J. (1995). Genetic algorithms for estimation problems with multiple optima, nondifferentiability, and other irregular features.
\textit{Journal of Business and Economic Statistics} {\bf 13} 53--66.

\item{}
Efron, B., Hastie, T., Johnstone, I., and Tibshirani, R. (2004). Least angle regression. \textit{The Annals of Statistics} \textbf{32} 407--451.

\item{}
Fan, J. and Chen, J. (1999). One-step local quasi-likelihood estimation. \textit{Journal of the Royal Statistical Society, Ser. B} \textbf{61} 927--943.

\item{}
Fan, J. and Gijbels, I. (1996). \textit{Local Polynomial Modeling and Its Applications}. Chapman \& Hall, London.

\item{}
Fan, J. Han, F., and Liu, H. (2013). Challenges of Big Data analysis. arXiv preprint, arXiv: 1308.1479v1.

\item{}
Fan, J. and Li, R. (2001). Variable selection via nonconcave penalized likelihood and its oracle properties. \textit{Journal of the American Statistical Association}
\textbf{96} 1348--1360.

\item{}
Fan, J. and Lv, J. (2008). Sure independence screening for ultrahigh dimensional feature space (with discussion). \textit{Journal of the Royal Statistical Society, Ser.
B} \textbf{70} 849--911.




\item{}
Fang, K. T., Hickernell, F. J., and Winker, P. (1996). Some global optimization algorithms in statistics, in \textit{Lecture Notes in Operations Research}, eds. by Du,
D. Z., Zhang, X. S. and Cheng, K. World Publishing Corporation, 14--24.

\item{}
Fang, K. T., Li, R., and Sudjianto, A. (2006). \textit{Design and Modeling for Computer Experiments}. Chapman \& Hall, London.

\item{}
Fang, K. T., Lin, D. K. J., Winker, P., and Zhang, Y. (2000). Uniform design: theory and application. \textit{Technometrics} {\bf 42} 237--248.

\item{}
Friedman, J., Hastie, T., and Tibshirani, R. (2000). Additive logistic regression: A statistical view of boosting (with discussion). \textit{The Annals of Statistics}
\textbf{28} 337--407.

\item{}
Freund, Y. and Schapire, R. (1997). A decision-theoretic generalization of on-line learning and an application to boosting. \textit{Journal of Computer and System
Sciences} {\bf 55} 119--139.

\item{}
Gan, L. and Jiang, J. (1999). A test for global maximum. \textit{Journal of the American Statistical Association} \textbf{94} 847--854.


\item{}
Gelman, A., Carlin, J.B., Stern, H.S., and Rubin, D.B. (2004). \textit{Bayesian Data Analysis, 2nd Edition}. Chapman \& Hall/CRC, New York.


\item{}
Gleser, L. J. (1965). On the asymptotic theory of fixed-size sequential confidence bounds for linear regression parameters. \textit{The Annals of Mathematical
Statistics} {\bf 36} 463--467.

\item{}
Hall, P. and Miller, H. (2009). Using generalized correlation to effect variable selection in very high dimensional problems. \textit{Journal of Computational and
Graphical Statistics} \textbf{18} 533--550.

\item{}
Hastie, T., Tibshirani, R., and Friedman, J. H. (2008). \textit{The Elements of Statistical Learning: Data Mining, Inference, and Prediction, 2nd Edition.} Springer, New
York.

\item{}
Hawkins, D. M. and Olive, D. J. (2002). Inconsistency of resampling algorithms for high-breakdown regression estimators and a new algorithm. \textit{Journal of the
American Statistical Association} {\bf 97} 136--148.

\item{}
Hofmann, M., Gatu, C., and Kontoghiorghes, E. J. (2007). Efficient algorithms for computing the best subset regression models for large-scale problems.
\textit{Computational Statistics} \& \textit{Data Analysis} {\bf 52} 16--29.

\item{}
Huber, P. J. (1981). \textit{Robust Statistics}. Wiley, New York.

\item{}
Hubert, M., Rousseeuw, P. J., and Van Aelst, S. (2008). High-breakdown robust multivariate methods. \textit{Statistical Science} {\bf 23} 92--119.


\item{}
Johnson, M. E., Moore, L. M., and Ylvisaker, D. (1990). Minimax and maximin distance designs. \textit{Journal of Statistical Planning and Inference} {\bf 26} 131--148.

\item{}
Kirkpatrick, S., Gelatt, C. D., and Vecchi, M. P. (1983). Optimization by simulated annealing. \textit{Science} \textbf{220} 671--680.

\item{}
Knight, K. and Fu, W. (2000). Asymptotics for lasso-type estimators. \textit{The Annals of Statistics} \textbf{28} 1356--1378.

\item{}
Ledoux, M. and Talagrand, M. (1980). \textit{{Probability in Banach Space: Isoperimetry and Processes}}. Springer, New York.

\item{}
Li, G., Peng, H., Zhang, J., and Zhu, L. (2012). Robust rank correlation based screening. \textit{The Annals of Statistics} \textbf{40} 1846--1877.

\item{}
Lindsay, B. G. (1994). Efficiency versus robustness: the case for minimum Hellinger distance and related methods. \textit{The Annals of Statistics} \textbf{22}
1081--1114.

\item{}
Loh, P.-L. and Wainwright, M. (2013). Regularized M-estimators with nonconvexity: Statistical and algorithmic theory for local optima. arXiv preprint, arXiv: 1305.2436.

\item{}
Lundy, M. and Mees, A. (1986). Convergence of an annealing algorithm. \textit{Mathematical Programming } \textbf{34} 111--124.

\item{}
Ma, P., Mahoney, M. W., and Yu, B. (2013). A statistical perspective on algorithmic leveraging. arXiv preprint arXiv:1306.5362v1

\item{}
Maronna, R. A., Martin, R.D., and Yohai, V.J. (2006). \textit{Robust Statistics: Theory and Methods}. Wiley, New York.

\item{}
Miller, A. (2002). \textit{{Subset Selection in Regression, 2nd Edition}}. Chapman \& Hall/CRC, New York.

\item{}
Niederreiter, H. (1992). \textit{Random Number Generation and Quasi-Monte Carlo Methods}. SAIM, Philadelphia.

\item{}
Owen, A. B. (2001). \textit{Empirical Likelihood}. Chapman \& Hall/CRC, London.

\item{}
Rao, C. R. and Wu, Y. (1989). A strongly consistent procedure for model selection in a regression problem. \textit{Biometrika} {\bf 76} 369--374.

\item{}
Rousseeuw, P. J. (1984). Least median of squares regression. \textit{Journal of the American Statistical Association} {\bf 79} 871--880.

\item{}
Rousseeuw, P. J. (1985). Multivariate estimation with high breakdown point. In \textit{Mathematical Statistics and Applications}, B, eds. by W. Grossmann, G. Pflug, I.
Vincze and W.Wertz. Reidel Publishing Company, Dordrecht.

\item{}
Rousseeuw, P. J. and Leroy, A.M. (1987). \textit{Robust Regression and Outlier Detection}. Wiley, New York.

\item{}
Rousseeuw, P. J. and Van Driessen, K. (1999). A fast algorithm for the minimum covariance determinant estimator. \textit{Technometrics} {\bf 41} 212--223.

\item{}
Rousseeuw, P. J. and Van Driessen, K. (2006). Computing LTS regression for large data sets. \textit{Data Mining and Knowledge Discovery} {\bf 12} 29--45.

\item{}
Schwarz, G. (1978). Estimating the dimension of a model. \textit{The Annals of Statistics} \textbf{6} 461--464.

\item{}
Scholkopf, B. and Smola, A. J. (2002). \textit{Learning With Kernels: Support Vector Machines, Regularization, Optimization, and Beyond}. MIT Press, Cambridge.

\item{}
Shao J. (1997). An asymptotic theory for linear model selection. \textit{Statistica Sinica} {\bf 7} 221--264.

\item{}
Tibshirani, R. (1996). Regression shrinkage and selection via the lasso. \textit{Journal of the Royal Statistical Society, Ser. B}. \textbf{58}, 267--288.

\item{}
Tibshirani, R., Bien, J., Friedman, J., Hastie, T., Simon, N., Taylor, J., and Tibshirani, R. J. (2012). Strong rules for discarding predictors in lasso-type problems.
\textit{Journal of the Royal Statistical Society, Ser. B}. \textbf{74}, 245--266.

\item{}
Van der Vaart, A. W. (1998). \textit{Asymptotic statistics}. Cambridge University Press, Cambridge.

\item{}
Wahba, G. (1990). \textit{Spline Models for Observational Data}. SIAM, Philadelphia.

\item{}
Wainwright, M. J. and Jordan, M. I., (2008). Graphical models, exponential families, and variational inference. \textit{Foundations and Trends in Machine Learning} {\bf
1} 1--305.

\item{}
Wald, A. (1949). Note on the consistency of maximum likelihood estimate. \textit{The Annals of Mathematical Statistics} {\bf 20} 595--601.

\item{}
Wang, H. (2009). Forward regression for ultra-high dimensional variable screening. \textit{Journal of the American Statistical Association} \textbf{104} 1512--1524.

\item{}
Wang, Z., Liu, H., and Zhang, T. (2013). Optimal computational and statistical rates of convergence for sparse nonconvex learning problems. arXiv preprint, arXiv:
1306.4960.

\item{}
Wendland, H. (2005). \textit{Scattered Data Approximation}. Cambridge University Press, Cambridge.

\item{}
Wolfowitz, J. (1957) The minimum distance method. \textit{The Annals of Mathematical Statistics} {\bf 28} 75--88.

\item{}
Wu, C. F. J. and Hamada, M. S. (2009). \textit{Experiments: planning, analysis, and optimization, 2nd edition}. Wiley, New York.

\item{}
Wu, J., Karunamuni, R., and Zhang, B. (2012) Efficient Hellinger distance estimates for semiparametric models. \textit{Journal of Multivariate Analysis} {\bf 107} 1--23.

\item{}
Xiong, S. (2014). Better subset regression. \textit{Biometrika} {\bf 101} 71--84.

\item{}
Xiong, S., Dai, B., and Qian, P. Z. G. (2013). OEM for least squares problems. arXiv preprint, arxiv: 1108.0185v2.


\item{}
Zhang. C-H. (2010). Nearly unbiased variable selection under minimax concave penalty. \textit{The Annals of Statistics} {\bf 38} 894--942.

\item{}
Zou, H. and Li, R. (2008). One-step sparse estimates in nonconcave penalized likelihood models. \textit{The Annals of Statistics} {\bf 36} 1509--1533.

\item{}
Zuo, Y. and Serfling, R. (2000). General notions of statistical depth function. \textit{The Annals of Statistics} {\bf 28} 461--482.

\end{description}}

\end{document}